\documentclass[a4paper,11pt]{amsart}
\usepackage[latin1]{inputenc}
\usepackage{amsfonts,amssymb,amsthm,amsmath,amsrefs}
\usepackage{hyperref} 
\usepackage{color}
\usepackage{tikz}
\usepackage{csquotes}
\usepackage{comment}

\numberwithin{equation}{section}

\newtheorem{thm}{Theorem}[section]
\newtheorem{cor}[thm]{Corollary}
\newtheorem{lem}[thm]{Lemma}
\newtheorem{prop}[thm]{Proposition}
\newtheorem{rem}[thm]{Remark}

\theoremstyle{definition}

\newtheorem{defn}[thm]{Definition}

\theoremstyle{remark}

\allowdisplaybreaks

\DeclareMathOperator{\op}{op}

\DeclareMathOperator{\re}{Re}

\DeclareMathOperator{\rang}{range}

\DeclareMathOperator{\supp}{supp}

\DeclareMathOperator{\image}{im}
\DeclareMathOperator{\ind}{ind-lim}
\DeclareMathOperator{\proj}{proj-lim}
\def\dbar{{\,\mathchar'26\mkern-12mu d}}

\def\skp#1{\langle#1\rangle}
\makeglossary

\title[ $H^\infty$-calculus for a Degenerate  Boundary Value Problem]{Bounded $H^\infty$-calculus for a Degenerate Elliptic Boundary Value Problem}
\author{Thorben Krietenstein}

\author{Elmar Schrohe}
\address{Institut für Analysis, Leibniz Universität Hannover, Welfengarten 1, 30167 Hannover, Germany}
\email{krietenstein@math.uni-hannover.de, schrohe@math.uni-hannover.de}
\begin{document}

\maketitle
\begin{abstract}
On a manifold $X$ with boundary and bounded geometry we consider a strongly elliptic second order operator $A$ together with a 
degenerate boundary operator $T$ of the form 
$T=\varphi_0\gamma_0 + \varphi_1\gamma_1$. 
Here $\gamma_0$ and $\gamma_1$ denote the evaluation of a function and its exterior normal derivative, respectively, at the boundary. We assume that $\varphi_0,\varphi_1\in C^{\infty}_b(\partial X)$, $\varphi_0, \varphi_1\ge 0$, and $\varphi_0+\varphi_1\geq c$, for some $c>0$. We also assume that the highest order coefficients of $A$ belong to $C^\tau(X)$ for some $\tau>0$ and the lower order coefficients are in $L_\infty(X)$.
We show that the $L_p(X)$-realization of $A$ with respect to the boundary operator  $T$ has a bounded $H^\infty$-calculus.
\end{abstract} 

\tableofcontents

\section{Introduction}
Maximal regularity has become an indispensable tool in the analysis 
of evolution equations as it can be used to establish in an uncomplicated 
way the existence 
of short time solutions to a large class of quasilinear parabolic problems. 
Maximal regularity in turn is implied by the existence of a bounded 
$H^\infty$-calculus, a concept introduced by McIntosh in 1986, \cite{McIntosh1986},
of angle $<\pi/2$.
Many elliptic operators are known to have a bounded $H^\infty$-calculus, see
e.g. Amann, Hieber, Simonett \cite{Amann1994} for the case of 
differential operators. Already in 1971 Seeley \cite{Seeley1971} had shown that
differential boundary value problems have bounded imaginary powers, 
a property which is very close to that of having a bounded $H^\infty$-calculus and
can often be shown by the same methods. 
Ellipticity, however, is not necessary in this context as shown in \cite{Bilyj10}; 
a hypoellipticity condition in the spirit of H\"ormander's conditions (4.2)' and (4.4)' in 
\cite{Hoermander67} is  sufficient.  
In the present article, we establish the existence of a bounded $H^\infty$-calculus
for a degenerate elliptic boundary value problem. 
We consider a strongly elliptic operator $A$, endowed with a boundary 
operator that, in general, will not satisfy the Lopatinsky-Shapiro ellipticity condition. 
The key point of our analysis then is the construction of a parameter-dependent 
parametrix to the resolvent with the help of  Boutet de Monvel's calculus 
for boundary value problems \cite{BoutetMonvel1971}. 
As a consequence of the non-ellipticity, however, this parametrix will only 
belong to an extended version of Boutet de Monvel's calculus that we sketch,
below. Still, this will enable us to deduce the necessary estimates for the 
existence of the bounded $H^\infty$-calculus.

Here are the details. 
Let $X$ be a manifold with boundary 
$\partial X$ and bounded geometry. Let $A$ be a strongly elliptic second order partial differential operator
on $X$ which in local coordinates can be written in the form  
\begin{align}\label{eq:A}
A=\sum_{1\leq k,l\leq n}a^{kl}(x) D_k D_l+\sum_{1\leq k\leq n}b^k(x)D_k+c^0(x),
\end{align}
where  $a^{kl}\in C^\tau({X})$ are real-valued,  
the matrix $(a^{kl}(x))_{1\leq k,l\leq n}$ is positive definite with a uniform positive lower bound, $b^k,c^0\in L_\infty(X)$, and $D_k = -i\partial_{x_k}$. 
Furthermore, we say that $A$ is $M$-elliptic, if all the norms of the coefficients are bounded by $M>0$ and the positive lower bound of the matrix is given by $1/M$. 
Obviously, this is no restriction as every operator as above is $M$-elliptic for some $M$.  
The operator $A$ is endowed
with a boundary operator $T$ of the form 
\begin{align}\label{eq:T}
T=\varphi_0\gamma_0+\varphi_1 \gamma_1.
\end{align}
Here $\gamma_0$ denotes the trace operator and $\gamma_1$ the exterior normal 
derivative at $\partial X$. 
Moreover, $\varphi_0, \varphi_1\in C_b^\infty(\partial X)$ 
are real-valued functions on the boundary with  $\varphi_1\geq0$ 
and $\varphi_0 + \varphi_1\geq c>0$.
We obtain the classical Dirichlet problem for $\varphi_0=1, \varphi_1=0$. 
The choice $\varphi_0=0, \varphi_1=1$ yields Neumann boundary conditions, 
and Robin problems correspond to the case where $\varphi_1$ is nowhere zero.

For given functions $f$ and $\phi$ we consider the boundary value problem with spectral parameter $\lambda$
\begin{align*}
(A-\lambda)u=f \text{ in } X, \quad 
Tu=\phi \text{ on }	\partial X,
\end{align*}
in $L_p(X)$,   $1<p<\infty$. 
To this end we introduce the  $L_p$-realization of the above boundary value problem, 
i.e. the unbounded operator $A_T$, acting like $A$ on the domain 
\begin{align*}
\mathcal{D}(A_T):=\{u\in L_p(X):Au\in L_p(X),\;Tu=0 \text{ on } \partial X\}.
\end{align*}
This problem has been investigated by many authors, see e.g.  
Egorov-Kondrat'ev \cite{EgorovKondratev69},  Kannai \cite{Kannai76} or Taira 
\cite{Taira76}, \cite{TairaCUP}, \cite{Taira16}, 
also for the case where the boundary 
operator $T$ involves an additional first order 
tangential differential operator. 
This makes the analysis more subtle and will be treated in a subsequent publication.
\medskip

We recall the notion of sectoriality: 
\begin{defn}
A closed and densely defined operator $B:\mathcal{D}(B)\in E\rightarrow E$, 
acting in a Banach space $E$ that is injective with dense range is called 
sectorial of type $\omega<\pi$, if for every $\omega<\theta<\pi$ 
there exists a constant $C_\theta$, such that
\begin{align*}
\sigma(B)\subset\Sigma_\theta \;\text{and}\; \|\lambda(B-\lambda)^{-1}\|_{\mathcal{L}(X)}\leq C_\theta \;\text {for all}\; \lambda\in \mathbb{C}\backslash\Sigma_\theta. 
\end{align*}
Here $\Sigma_\theta=\{\lambda\in\mathbb{C}\backslash\{0\}:|\arg(\lambda)|\leq \theta\}\cup \{0\}$ is the sector of angle $\theta$ around the positive real axis.
\end{defn}
It has been shown by Taira that, for a bounded domain $X$, 
the $L_p$-realization $A_T$ is sectorial of type $\varepsilon$ for every $\varepsilon>0$, possibly after replacing $A$ by $A+c$ for a positive constant $c$. 
In particular, it generates an  analytic semigroup. 
For details see e.g.  \cite[Theorem 1.2]{TairaCUP}.

\subsection*{Bounded $H^\infty$ calculus}

By $H^\infty(\Sigma_\theta)$ we denote the space of 
bounded holomorphic functions in the interior of the sector $\Sigma_\theta$ 
and by $H_*^\infty(\Sigma_\theta)$ the subspace of all  functions $f$  
such that  $|f(\lambda)|\leq C(|\lambda|^\epsilon+|\lambda|^{-\varepsilon})^{-1}$ 
for suitable $C,\varepsilon>0$. 
It is well-known that this is a dense subspace with respect to the topology of uniform convergence on compact sets.

For a sectorial operator $B$ of type $\omega$, $\theta'\in\  ]\omega,\theta[$ and $f\in H_*^\infty(\Lambda_\theta)$
let
\begin{align*}
f(B)=\frac i{2\pi}\int _{\partial \Lambda_{\theta'}} f(\lambda)(B-\lambda)^{-1}\,d\lambda\in \mathcal{L}(E) .
\end{align*}
The integral exists due to the sectoriality  and is independent of the choice of $\theta'$ by Cauchy's integral theorem. Given $f\in H^\infty(\Sigma_\theta)$, we can approximate $f$ by a sequence  $(f_n)\subset H_*^\infty(\Sigma_\theta)$ and define 
\begin{align*}
f(B)x:=\lim f_n(B)x \; \text {for}\; x\in \mathcal{D}(B)\cap \rang(B). 
\end{align*}
It can be shown that $\mathcal{D}(B)\cap \rang(B)$ is dense in $E$ and that the above equation defines a closable operator. The closure is again denoted by $f(B)$.

\begin{defn}
We say that a sectorial operator $B$ of type $\omega$ admits a bounded $H^\infty$ calculus of angle $\omega$, if for any $\omega<\theta<\pi$ there exists a constant $C_\theta>0$, such that 
\begin{align}
\label{eq:Hinfty_estimate}
\|f(B)\|_{\mathcal{L}(E)}\leq C_\theta\|f\|_\infty , \quad f\in H^\infty(\Sigma_{\theta}).
\end{align}
\end{defn}

According to the  principle of uniform boundedness it is sufficient to verify estimate \eqref{eq:Hinfty_estimate} for all $f\in H_*^\infty(\Sigma_{\theta})$.

\subsection*{Main results}
\begin{thm}\label{thm_hinfty}
\label{thm:main_result} Let $(X,g)$ be a manifold with boundary and bounded geometry. Let $T$ be as in \eqref{eq:T} and $A_T$ be the realization given above of an $M$-elliptic sufficiently regular second order differential operator. Then, for every $0<\vartheta<\pi$ a constant $\nu=\nu(M,|t|_*,\vartheta)\geq0$ exists such that $A_T+\nu$ allows an $H^\infty(\Sigma_\vartheta)$-calculus in $L_p(X)$. Moreover, a constant $C=C(M,|t|_*,\vartheta)>0$ exists such that for all $f\in H^\infty(\Sigma_\vartheta)$ the following estimate holds:
\begin{align*}
\|f(A_T)\|_{\mathcal{B}(L_p(X))}\leq C\|f\|_{\infty}.
\end{align*}
\end{thm}

As a corollary, we obtain unique solvability for the full boundary value problem. For this we need some notation. As before, $(X,g)$ is 
a manifold with boundary and bounded geometry, $1<p<\infty$.

We denote by  $B^{s}_p(\partial X):= B^s_{p,p}(\partial X)$ the $L_p$-Besov space of order $s\in \mathbb R$ on $\partial X$ as defined in \cite{Grosse2013}. 
According to \cite[Theorem 4.10]{Grosse2013}, $B^{s-1/p}_p(\partial X)$, $1<p<\infty$, $s>1/p$, coincides with the space of all restrictions to $\partial X$ of  functions in $H^s_p(X)$.
The theorem, below, can be shown by modifying the proof of \cite[Theorem 4.10]{Grosse2013} in the spirit of the proof of \cite[Theorem 2.9.2]{Triebel1978}. 

\begin{thm}\label{thm_ext}Let $s>1+1/p$. Then, given $v_0\in
B^{s-1/p}_p(\partial X)$ and $v_1\in B^{s-1-1/p}_p (\partial X)$ there exists 
$u\in H^s_p(X)$ such that $\gamma_0u=v_0$ and $\partial_\nu u = v_1$.
\end{thm} 

\begin{defn} 
For $s\in \mathbb R$ and  the boundary condition $T$ in \eqref{eq:T}, we define  
$$B^{s-1-1/p}_{p,T}(\partial X)=\{v= \varphi_0 v_0+\varphi_1v_1\mid v_0\in B^{s-1/p}_p(\partial X), v_1\in B^{s-1-1/p}_p(\partial X) \} .
$$
\end{defn} 
Clearly, this is a Banach space with the topology of the non-direct sum.  

\begin{prop}
For $s>1+1/p$ the mapping $T: H^s_p(X) \to B^{s-1-1/p}_{p,T}(\partial X)$ is surjective. 
\end{prop}
In fact, given $v= \varphi_0 v_0+\varphi_1v_1$ in $B^{s-1-1/p}_{p,T}(\partial X)$, Theorem \ref{thm_ext} implies that we find $u_0$ and $u_1$ in $H^s_p(X)$ such that  
$\gamma_0u_0 = v_0$, $\gamma_1u_0=0$, $\gamma_0u_1=0$ and $\gamma_1u_1=v_1$.
Then $u_0+u_1$ is a preimage of $v$ under $T$. 
\begin{thm}\label{thm_iso}For very $0<\vartheta<\pi$ the operator 
\begin{eqnarray}\label{eq_iso}
\binom{A-\lambda}{T}: H^2_p(X)\longrightarrow \begin{array}{c} L_p(X) \\\oplus\\ B_{p,T}^{1-1/p}(\partial X)\end{array} 
\end{eqnarray}
is a topological isomorphism for $\lambda \in \Sigma_\theta$, $|\lambda|$ sufficiently large.    
\end{thm}
This is immediate from Theorem \ref{thm_hinfty} and the surjectivity of $T$: Given 
$f\in L_p(X)$ and $v \in B^{1-1/p}_{p,T}(\partial X)$, we first fix $w_0\in H^2_p(X)$ with $Tw_0=v$. By Theorem \ref{thm_hinfty},  the problem 
$(A-\lambda)w=f-(A-\lambda)w_0$, $Tu=0$ has a unique solution $w\in H^2_p(X)$.
Then $u=w+w_0$ is the (unique) solution to $(A-\lambda)u=f$, $T u = v$.   
Hence \eqref{eq_iso} is a bijection. As it is continuous, it is a topological isomorphism in view of the closed graph theorem.
\medskip

Finally, we apply our results to the porous medium equation with time-independent boundary condition $T$ and strictly positive initial value. Details can be found in Section 
\ref{sec:porus_medium_equation}. We obtain: 

\begin{thm}\label{thm_PME} Let $1<p,q<\infty$, $n/p+2/q<1$, $m>0$, $v_0\in H^2_p(X)$ with $v_0\ge c>0$, and $\phi=Tv_0$. Then the porous medium equation
\begin{align}\label{eq:PME_1}
\begin{cases}
\dot{v}-\Delta_g v^m=0\\
Tv=\phi\\
v\vert_{t=0}=v_0
\end{cases}.
\end{align}
has a unique short time solution of maximal regularity, i.e. there exist $\tau>0$ and a unique solution 
\begin{align*}
v\in L_q(0,\tau;H^2_p(X)\cap\{Tv=Tv_0\})\cap W^1_q(0,\tau;L_p(X))
\end{align*}
of the porous medium equation \eqref{eq:PME_1}.
\end{thm}

\subsection*{Relation to previous work} In \cite{Abels2005}, Abels developed a (different) variant of Boutet de Monvel's calculus with non-smooth symbols in order to construct parametrices to 
elliptic operators with Hölder regularity. 

Theorem \ref{thm_hinfty} extends \cite[Theorem 1.2]{TairaCUP} in that (i) one can now treat manifolds of bounded geometry instead of bounded domains, (ii) the differentiability assumptions on the coefficients are reduced from $C^\infty$ to $C^\tau$, $\tau>0$, for the top order terms and $L^\infty$ for the lower order terms and  (iii) one obtains the existence of a bounded $H^\infty$-calculus rather than the existence of a holomorphic semigroup.  
Theorem \ref{thm_iso} extends \cite[Theorem 1.1]{TairaCUP} to the case of manifolds with boundary and bounded geometry and operators with non-smooth coefficients, with the restriction that $\lambda\in \Sigma_\vartheta$ has to  be large and we work on $H^2_p(X)$ as a consequence of the non-smoothness of the coefficients. 
 
In \cite{Taira2014} and \cite{Taira2020} Taira treats more general Waldenfels integro-differential operators to which the present methods should also be applicable.

While there is a wealth of literature on the porous medium equation, it seems to be new to study it  on manifolds of bounded geometry and with degenerate boundary condition.

\subsection*{Outline of the paper}
In order to establish \eqref{eq:Hinfty_estimate} for $A_T+c$ we have to show that  
for every fixed $0<\theta<\pi$
\begin{align}
\label{eq:ess_est}
\left\|\int_{\partial \Lambda_\theta} f(\lambda)(A_T+c-\lambda)^{-1}\,d\lambda\right\|_{\mathcal{L}(L_p(X))}\leq C_\theta \|f\|_\infty, \quad f\in H_*^\infty(\Lambda_\theta).
\end{align}

It is clear that a good understanding of $(A_T+c-\lambda)^{-1}$ on the rays $\arg\lambda = \pm \theta$, $0<\theta<\pi$  is essential for this task.

The main tool we use in this paper is Boutet de Monvel's calculus for boundary value 
problems \cite{BoutetMonvel1971}. Details can be found e.g.\ in the monographs by 
Rempel and Schulze \cite{Rempel1982} and Grubb \cite{Grubb1996} 
or in the short introduction \cite{Schrohe2001}.
We will also need a slight generalization for which  details will be given below. 
Recall that an operator of order $m\in \mathbb R$ and class (or type) $d \in \mathbb N_0$ in Boutet de Monvel's calculus on 
$\mathbb R^n_+$ is a matrix of operators
\begin{align*}
\begin{pmatrix}
P_++G& K\\
T & S
\end{pmatrix}
:
\begin{matrix}
\mathcal S(\mathbb R^n_+,E_0)\\\oplus\\\mathcal S(\mathbb R^{n-1} ,F_0)
\end{matrix}
\rightarrow
\begin{matrix}
\mathcal S(\mathbb R^n_+,E_1)\\\oplus\\\mathcal S( \mathbb R^{n-1},F_1)
\end{matrix}.
\end{align*}
Here $E_0$ and $E_1$ are vector bundles over $\mathbb R^n$, 
and $F_0$, $F_1$ are vector bundles over $\partial \mathbb R^n_+=\mathbb R^{n-1}$. 
Moreover, $P$ is a pseudodifferential operator on $\mathbb R^n$ satisfying the transmission condition,  and $P_+$ denotes its truncation to $\mathbb R^n_+$: 
$P_+ = r^+Pe^+$, where $e^+$ denotes extension by zero from $\mathcal S(\mathbb R^n_+, E_0)$ to, say, $L_2( \mathbb R^n,E_0)$, and $r^+$ denotes the restriction of distributions on $\mathbb R^n$ to those on $\mathbb R^n_+$.   
The operators $G$ and $T$ are singular Green and trace operators of order 
$m$ and class $d$, respectively; $K$ is a potential operator of order $m$.  
Finally $S$ is a pseudodifferential operator on the boundary $\mathbb R^{n-1}$ 
of order $m$.

Boutet de Monvel's calculus is closed under compositions provided  the vector bundles fit together. Via coordinate maps the calculus can be transferred to smooth manifolds with boundary. 

Boutet de Monvel's calculus  has a symbolic structure with a notion of ellipticity, and there exist parametrices to elliptic elements in the calculus. Moreover, the calculus contains its inverses whenever these exist. An operator of order $m$ and class $d$ as above extends to a bounded map
$$
H^{s+m}_p(X,E_0)\oplus B_p^{s+m-1/p}(\partial X,F_0)
\rightarrow
H^{s}_p(X,E_1)\\\oplus\\B_p^{s-1/p}(\partial X,F_1),
$$ 
provided $s>d-1+1/p$, where $H^s_p$ denotes the usual Sobolev space and $B^s_p=B^s_{p,p}$ the Besov
space of order $s$, see  Grubb \cite{Grubb1990}. 
\medskip

It is well-known that the operator 
$$\binom{(A-\lambda)_+}{\gamma_0}: 
H^2_p(X) \to 
\begin{matrix} L_p(X)\\\oplus\\ B^{2-1/p}_p(\partial X)\end{matrix}$$
is invertible for $\lambda \in \Lambda_\theta$,  $\theta>0$, $|\lambda|$ sufficiently large,
whenever $X$ is a compact manifold with boundary or $\mathbb R^n_+$.  
In particular, it is invertible for all $\lambda\in\Lambda_\theta$, if we replace  $A$ by 
$A+c$ for $c>0$ sufficiently large.  
In order to keep the notation simple, we will assume from now on that $A$ has been
replaced by $A+c$ for such $c$ and write $A$ instead of $A+c$.
 
Apart from the fact that $\gamma_0$ is formally not of the right order (which is of no importance here and can be easily be arranged),  the problem fits into Boutet de Monvel's 
calculus and one obtains the inverse in the form 
$$ 
 \binom{(A-\lambda)_+}{\gamma_0}^{-1} = (( A-\lambda)_+^{-1} + G^D_{\lambda} \quad K^D_{\lambda}).
$$
Here $(A-\lambda)^{-1}$ is the resolvent on a closed manifold with boundary into which $X$ embeds (in case $X$ is compact) or on $\mathbb R^n$ (if $X = \mathbb R_+$),  see \cite{Grubb1996}.

We will denote the corresponding truncation by $Q_{\lambda,+}$:  
$$Q_{\lambda,+} = ((A-\lambda)^{-1})_+ .$$ 
As a consequence,
\begin{eqnarray*}
\binom{(A-\lambda)_+}{T}  (Q_{\lambda,+} + G^D_{\lambda} \quad K^D_{\lambda})
= 
\begin{pmatrix} 
I&0\\
T (Q_{\lambda,+} + G^D_{\lambda} )& TK_\lambda^D
\end{pmatrix}
\end{eqnarray*} 
Assuming that $S_\lambda:=TK_\lambda^D$ is invertible with inverse $S^{-1}_\lambda$, we find that 
\begin{eqnarray}
\lefteqn{\binom{(A-\lambda)_+}{T}^{-1} }\nonumber\\
&=&(Q_{\lambda,+} + G^D_{\lambda} - K_\lambda^DS^{-1}_\lambda T(Q_{\lambda,+}+G^D_\lambda)\quad 
K_\lambda^D S^{-1}_\lambda)
\end{eqnarray}
For the realization $(A-\lambda)_T$ we obtain:
\begin{eqnarray*}\lefteqn{(A-\lambda)_T^{-1} }\nonumber\\
&=&Q_{\lambda,+} + G^D_{\lambda} - K_\lambda^DS^{-1}_\lambda T(Q_{\lambda,+}+G^D_\lambda)\\
&=&
Q_{\lambda,+} + G^D_{\lambda} +G_\lambda^T
\end{eqnarray*}
with 
\begin{align}
\label{eq:def_GT}
G_\lambda^T := - K_\lambda^DS^{-1}_\lambda T(Q_{\lambda,+}+G^D_\lambda).
\end{align}

\begin{lem}\label{lem:q+g}
For every choice of $\theta\in\ ]0,\pi[$, there exists a constant $C_\theta\ge0$ such that  
$$\left\|\int_{\partial \Lambda_\theta} f(\lambda)(Q_{\lambda,+}+G_\lambda^D)\,d\lambda\right\|_{\mathcal{L}(L_p(X))}\leq C_\theta \|f\|_\infty \;\text{for all}\; f\in H_*^\infty(\Lambda_\theta).$$
\end{lem} 

Lemma \ref{lem:q+g} is well-known, the proof relies on the fact that the operators $Q_\lambda$ and 
$G_\lambda^D$ are parameter-dependent operators of order $-2$ in Boutet de Monvel's calculus, if one writes $-\lambda = \mu^2e^{i\theta}$ and considers 
$Q_\lambda$ and $G^D_\lambda$ as functions of $\mu$, see e.g.\ Grubb \cite{Grubb1996}. For the more general situation of a manifold with boundary and conic singularities, 
see \cite{CSS2007}.

It remains to study the term $G^T_\lambda$. 
It will turn out that $TK_\lambda^D$ is a hypoelliptic pseudodifferential operator of order $1$ on the boundary. 
As we will see, it has a parametrix with local symbols in the Hörmander class $S^0_{1,1/2}$ which then agrees with $S^{-1}_\lambda$ up to a regularizing operator. In order to treat the composition of $S^{-1}_\lambda$  with the operators $K_\lambda^D$ and 
$Q_{\lambda,+}+G^D_\lambda$, we will need an extension of the classical Boutet de Monvel calculus. 

\textbf{Acknowledgement.} The results in this article are based on the second authors thesis, see \cite{Krietenstein2019}. The authors thank K. Taira for helpful discussions.

\section{An Extended Boutet de Monvel Type Calculus}
We recall the algebra $\mathcal H$ of functions on $\mathbb R$: It is the direct sum
$$\mathcal H = \mathcal H^+\oplus \mathcal H^-_{-1}\oplus \mathcal H', 
$$
where 
\begin{align*}
\mathcal{H}^+:=\{\mathcal{F}(e^+u):u\in \mathcal{S}(\mathbb{R}_+)\}, \quad \mathcal{H}_{-1}^-:=\{\mathcal{F}(e^-u):u\in \mathcal{S}(\mathbb{R}_-)\}
\end{align*}
and $\mathcal H'$ is the space of all polynomials on $\mathbb R$. 
The sum is direct, since the functions in $\mathcal H^+$ and $\mathcal H^-_{-1}$ 
decay to first order. 

It will be helpful to use also weighted Sobolev spaces on $\mathbb R_+$: 
For $\mathbf s=(s_1, s_2)\in \mathbb R^2$ we let $H^{\mathbf s}_p(\mathbb R_+) $ denote the space of all $u\in \mathcal D'(\mathbb R_+)$ such that $\skp{x}^{s_2}u$ belongs to the ordinary  Sobolev space 
$H^{s_1}_p(\mathbb R_+)$. 
We then have 
\begin{eqnarray}\label{eq:limits.1}
\mathcal S(\mathbb R_+)&=& \proj H_p^{\mathbf s}(\mathbb R_+) \text{ and  }\\
\label{eq:limits.2}
\mathcal S'(\mathbb R_+)&=&\ind  (H_p^{\mathbf s}(\mathbb R_+))' = \ind \dot H_{1-1/p}^{\mathbf s}(\mathbb  R_+),
\end{eqnarray} 
where the limits are taken over $\mathbf s\in \mathbb R^2$ and  
$\dot H_q^{\mathbf s}(\mathbb  R_+)$ denotes all distributions $u$ in 
$H_q^{\mathbf s}(\mathbb R)$ for which $\supp u\subset \overline{\mathbb R}_+$.

\subsection*{Operator-valued symbols}

Let $E,F$ be Banach spaces with strongly continuous group actions 
$\kappa^E_\lambda$, $\kappa^F_\lambda$, $\lambda>0$, as introduced by Schulze in \cite{Schulze1991}.
Given $q\in \mathbb N$, $m\in \mathbb R$, $0\le \delta\le 1$ $\delta <1$, 
we call a function $a=a(y,\eta)\in C^\infty(\mathbb R^q\times \mathbb R^q,\mathcal L(E,F))$
an operator-valued symbol in $S^m_{1,\delta}(\mathbb R^q\times \mathbb R^q; E,F)$
if, for all multi-indices $\alpha, \beta$ there exist constants $C_{\alpha, \beta}$ such that 
$$\|\kappa^F_{\skp{\eta}^{-1}}
D^\alpha_\eta D^\beta_y a(y,\eta) \kappa^E_{\skp\eta}\|_{\mathcal L(E,F)}
\le C_{\alpha, \beta} \skp{\eta}^{m-|\alpha|+\delta |\beta|}.$$

In the sequel, we will mostly have the case where and $E$ and $F$ are 
either $\mathbb C$ or (weighted) Sobolev spaces over $\mathbb R$ or 
$\mathbb R_+$. On $\mathbb C$ we will use the trivial group action; on the Sobolev
spaces we will use the action given by 
$\kappa_\lambda u(t) = \sqrt\lambda f(\lambda t)$. 
Note that this group action is unitary on $L^2(\mathbb R)$ and $L^2(\mathbb R_+)$. 
Using the representations \eqref{eq:limits.1} and \eqref{eq:limits.2}, 
the above definition extends to the case, where $E=\mathcal S(\mathbb R_+)$,
$E=\mathcal S'(\mathbb R_+)$ or $F=\mathcal S(\mathbb R_+)$,
see \cite{Schrohe2001} for details. 

\subsection*{The transmission condition}
\begin{defn} A symbol $p\in S^m_{1,\delta}(\mathbb{R}^n\times\mathbb{R}^n)$ satisfies the transmission condition at $x_n=0$ provided that, for all $k\in \mathbb{N}_0$
\begin{align*}
p_{[k]}(x',\xi',\xi_n):=(\partial^k_{x_n}p)(x',0,\xi',\langle\xi'\rangle\xi_n)\in S^{m+\delta k}_{1,\delta}(\mathbb{R}^{n-1}\times\mathbb{R}^{n-1})\hat{\otimes}\mathcal{H}
\end{align*}
We write 
$p\in \mathcal{P}^m_{1,\delta}(\mathbb{R}^{n-1}\times\mathbb{R}^{n-1})$.
\end{defn}

\begin{rem}$\mathcal{P}^\infty_{1,\delta}:=\bigcup_m{\mathcal{P}}^m_{1,\delta}$ is closed under the usual symbol operations, i.e. addition, pointwise multiplication and inversion, differentiation, Leibniz product and asymptotic summation.
We also have $S^{-\infty}=\bigcap_m{\mathcal{P}}^m_{1,\delta}$. 
\end{rem}

\begin{thm}\label{thm_transmission_normal} 
Let $p\in \mathcal{P}^m_{1,\delta}(\mathbb{R}^n\times\mathbb{R}^n)$.
Then
\begin{align*}
\op_n(p)_+\in 
S_{1,\delta}^{m}(\mathbb{R}^{n-1}\times\mathbb{R}^{n-1};
\mathcal S(\mathbb R_+), \mathcal S(\mathbb R_+)).
\end{align*}
\end{thm}

\begin{proof}
This follows from the fact that 
$\kappa_{\skp{\xi'}^{-1}}\op_n (p)\kappa_{\skp{\xi'}} = \op_n (q_{x',\xi'})$, 
where $q_{x',\xi'}(x_n,\xi_n) = p(x', x_n/\skp{\xi'}, \xi', \skp{\xi'}\xi_n)$ and the 
corresponding proof for H\"or\-man\-der type $(1,0)$; this is Theorem 2.12 in \cite{Schrohe2001}. The arguments carry over to general $(1,\delta)$.
\end{proof}

\subsection*{Potential, trace and singular Green symbols}
\begin{defn} Let $m\in \mathbb R$, $d\in \mathbb N_0$. All functions, below, may be matrix valued. 
\begin{itemize}
\item
A function $k\in C^\infty(\mathbb{R}^{n-1}\times\mathbb{R}^{n-1}\times\mathbb{R})$ belongs to the space $\mathcal{K}^m_{1,\delta}(\mathbb{R}^{n-1}\times\mathbb{R}^{n-1})$ of potential symbols of order $m$ and Hörmander type $(1,\delta)$, if
\begin{align*}
k_{[0]}(x',\xi';\xi_n):=k(x',\xi';\langle\xi'\rangle\xi_n)\in S^{m-1}_{1,\delta}(\mathbb{R}^{n-1}\times\mathbb{R}^{n-1})\hat{\otimes}\mathcal{H}^+_{\xi_n}.
\end{align*}
\item
A function $t\in C^\infty(\mathbb{R}^{n-1}\times\mathbb{R}^{n-1}\times\mathbb{R})$ 
belongs to the space 
$\mathcal{T}^{m,d}_{1,\delta}(\mathbb{R}^{n-1}\times\mathbb{R}^{n-1})$ 
of trace symbols of order $m$, class $d$ and Hörmander type $(1,\delta)$, if
\begin{align*}
t_{[0]}(x',\xi';\xi_n):=t(x',\xi';\langle\xi'\rangle\xi_n)\in S^{m}_{1,\delta}(\mathbb{R}^{n-1}\times\mathbb{R}^{n-1})\hat{\otimes}\mathcal{H}^-_{d-1}.
\end{align*}
\item A function $g\in C^\infty(\mathbb{R}^{n-1}\times\mathbb{R}^{n-1}\times\mathbb{R}\times\mathbb{R})$ belongs to the space
$\mathcal{G}^{m,d}_{1,\delta}(\mathbb{R}^{n-1}\times\mathbb{R}^{n-1})$
 of singular Green symbols of order $m$, class $d$ and Hörmander type $(1,\delta)$, if
\begin{align*}
g_{[0]}(x',\xi';\xi_n,\eta_n):=g(x',\xi';\langle\xi'\rangle\xi_n,\langle\xi'\rangle\eta_n)\in S^{m-1}_{1,\delta}(\mathbb{R}^{n-1}\times\mathbb{R}^{n-1})\hat{\otimes}\mathcal{H}^+_{\xi_n}\hat{\otimes}\mathcal{H}^{-}_{d-1,\eta_n}.
\end{align*}
\end{itemize}
\end{defn}
The spaces $\mathcal{K}^m_{1,0}$, $\mathcal{T}^m_{1,0}$ and $\mathcal{G}^m_{1,0}$ are denoted by Grubb in \cite{Grubb1996} as $S_{1,0}^{m-1}(\mathbb{R}^{n-1}\times\mathbb{R}^{n-1};\mathcal{H}^+)$, $S^{m}_{1,0}(\mathbb{R}^{n-1}\times\mathbb{R}^{n-1};\mathcal{H}^-_{d-1})$, and $S^{m-1}_{1,0}(\mathbb{R}^{n-1}\times\mathbb{R}^{n-1};\mathcal{H}^+\hat{\otimes}\mathcal{H}^-_{d-1})$. Rempel and Schulze denote them in \cite{Rempel1982} by $\mathfrak{K}^{m-1}(\mathbb{R}^{n-1}\times\mathbb{R}^n)$, $\mathfrak{T}^{m,d}(\mathbb{R}^{n-1}\times\mathbb{R}^n)$ and $\mathfrak{B}^{m-1,d}(\mathbb{R}^{n-1}\times\mathbb{R}^{n+1})$. They are Fréchet spaces with the topologies induced by the scaled functions.
For fixed $(x',\xi')$ the symbols above define Wiener-Hopf operators. Hence we obtain an action in the normal direction:
\begin{align*}
[\op_n k](x',\xi'):&=r^+\mathcal{F}_{\xi_n\rightarrow x_n}^{-1}k(x',\xi'):\mathbb{C}\rightarrow\mathcal{S}(\mathbb{R}_+),\\
[\op_n t](x',\xi'):&=I_{\xi_n}^+t(x',\xi';\xi_n)\mathcal{F}_{y_n\rightarrow \xi_n}e^+:\mathcal{S}(\mathbb{R}_+)\rightarrow\mathbb{C}\;\;\text{and}\\
[\op_n g](x',\xi'):&=r^+\mathcal{F}_{\xi_n\rightarrow x_n}^{-1}I^+_{\eta_n}g(x',\xi';\xi_n,\eta_n)\mathcal{F}_{y_n\rightarrow\eta_n}e^+:\mathcal{S}(\mathbb{R}_+)\rightarrow\mathcal{S}(\mathbb{R}_+),
\end{align*}
where $I^+$ is the plus-integral, see \cite[p.166]{Grubb1996}. 
We can interpret $\op_nk$, $\op_n t$ and 
$\op_ng$ as operator-valued symbols. Depending on the class there are several extensions possible.
\begin{thm}[Description by operator-valued symbols]\label{thm_operator_description}
Let $\mathbf{s}\in\mathbb{R}^2$ with $s_1>d-1/2$. 
The following maps are bounded and linear.
\begin{enumerate}
\item $\op_n:\mathcal{G}_{1,\delta}^{m,0}(\mathbb{R}^{n-1}\times\mathbb{R}^{n-1})\rightarrow S^{m}_{1,\delta}(\mathbb{R}^{n-1}\times\mathbb{R}^{n-1};\mathcal{S}'(\mathbb{R}_+){},\mathcal{S}(\mathbb{R}_+){})$
\item $\op_n:\mathcal{G}_{1,\delta}^{m,d}(\mathbb{R}^{n-1}\times\mathbb{R}^{n-1})\rightarrow S^{m}_{1,\delta}(\mathbb{R}^{n-1}\times\mathbb{R}^{n-1};H_2^\mathbf{s}(\mathbb{R}_+){},\mathcal{S}(\mathbb{R}_+){})$
\item $\op_n:\mathcal{K}_{1,\delta}^{m}(\mathbb{R}^{n-1}\times\mathbb{R}^{n-1})\rightarrow S^{m-1/2}_{1,\delta}(\mathbb{R}^{n-1}\times\mathbb{R}^{n-1};\mathbb{C},\mathcal{S}(\mathbb{R}_+){})$
\item $\op_n:\mathcal{T}_{1,\delta}^{m,0}(\mathbb{R}^{n-1}\times\mathbb{R}^{n-1})\rightarrow S^{m+1/2}_{1,\delta}(\mathbb{R}^{n-1}\times\mathbb{R}^{n-1};\mathcal{S}'(\mathbb{R}_+){},\mathbb{C})$
\item $\op_n:\mathcal{T}_{1,\delta}^{m,d}(\mathbb{R}^{n-1}\times\mathbb{R}^{n-1})\rightarrow S^{m+1/2}_{1,\delta}(\mathbb{R}^{n-1}\times\mathbb{R}^{n-1};H_2^\mathbf{s}(\mathbb{R}_+){},\mathbb{C})$
\end{enumerate}
\end{thm}
We omit the proof, which is straightforward. We also need the description via 
symbol-kernels. To this end we define
\begin{align*}
\widetilde{\mathcal{K}}^m_{1,\delta}:=\mathcal{F}_{\xi_n\rightarrow x_n}^{-1}\mathcal{K}^m_{1,\delta},\;\;\widetilde{\mathcal{T}}^{m}_{1,\delta}:=\overline{\mathcal{F}}_{\xi_n\rightarrow y_n}^{-1}\mathcal{T}^{m,0}_{1,\delta}\;\;\text{and}\;\;\widetilde{\mathcal{T}}^{m}_{1,\delta}:=\mathcal{F}^{-1}_{\xi_n\rightarrow x_n}\overline{\mathcal{F}}_{\eta_n\rightarrow y_n}^{-1}\mathcal{G}^{m,0}_{1,\delta}.
\end{align*}

\begin{thm}[Description by symbol-kernels]\label{thm_kernel_description}The following assertions hold:
\begin{itemize}
\item[(i)] For every operator-valued symbol $k\in S_{1,\delta}^m(\mathbb{R}^{n-1}\times\mathbb{R}^{n-1};\mathbb{C},\mathcal{S}(\mathbb{R}_+))$ there exists a unique $\tilde{k}\in \widetilde{\mathcal{K}}_{1,\delta}^m(\mathbb{R}^{n-1}\times\mathbb{R}^{n-1})$, such that $$[k(x',\xi')c](x_n)=\tilde{k}(x',\xi';x_n)c,\; 
c\in\mathbb{C} .$$
\item[(ii)] For every operator-valued symbol $t\in S_{1,\delta}^m(\mathbb{R}^{n-1}\times\mathbb{R}^{n-1};\mathcal{S}'(\mathbb{R}_+),\mathbb{C})$ there exists a unique $\tilde{t}\in \widetilde{\mathcal{T}}^{m,0}_{1,\delta}(\mathbb{R}^{n-1}\times\mathbb{R}^{n-1})$, such that $$t(x',\xi')u=\int_{\mathbb{R}_+}\tilde{t}(x',\xi';y_n)u(y_n) \,dy_n,\; 
u\in\mathcal{S}(\mathbb{R}_+).$$
\item[(iii)]For every operator-valued symbol $g\in S_{1,\delta}^m(\mathbb{R}^{n-1}\times\mathbb{R}^{n-1};\mathcal{S}'(\mathbb{R}_+),\mathcal{S}(\mathbb{R}_+))$ there exists a unique $\tilde{g}\in \widetilde{\mathcal{G}}^{m}_{1,\delta}(\mathbb{R}^{n-1}\times\mathbb{R}^{n-1})$, such that $$[g(x',\xi')u](x_n)=\int_{\mathbb{R}_+}\tilde{g}(x',\xi';x_n,y_n)u(y_n)\,dy_n,\; 
u\in\mathcal{S}(\mathbb{R}_+).$$ 
\end{itemize}
\end{thm}

\begin{proof}
See Theorems 3.7 and 3.9 in \cite{Schrohe2001}.
\end{proof}
\begin{cor}
The maps $(1), (3)$, and $(4)$ in Theorem \ref{thm_operator_description} 
are bijections. The maps $(2)$ and $(5)$ are bijections onto their image, which is the set
of all operators of the form
\begin{align*}
&\op_n g_0+\sum_{j=0}^{d-1}\op_n k_j \gamma^+_j,\  g_0 \in \mathcal G^{m,0}_{1,\delta} (\mathbb R^{n-1} \times \mathbb R^{n-1}), k_j\in \mathcal{K}^{m-j}_{1,\delta}(\mathbb{R}^{n-1}\times\mathbb{R}^{n-1})\;\;\text{resp.}\\
&\op_n t_0+\sum_{j=0}^{d-1} s_j \gamma^+_j, \  t_0\in 
\mathcal{T}^{m,0}_{1,\delta}(\mathbb{R}^{n-1}\times\mathbb{R}^{n-1}),
s_j\in S^{m-j}_{1,\delta}(\mathbb{R}^{n-1}\times\mathbb{R}^{n-1}).
\end{align*}
\end{cor}

\begin{proof}
We get from symbols to operator-valued symbols, to symbol-kernels, 
and back to symbols by Theorem \ref{thm_operator_description}, 
Theorem \ref{thm_kernel_description} and the Fourier transform. 
For non-zero class we use the fact $I^+\xi^j\mathcal{F}e^+\phi=(-i)^j\gamma_j^+\phi$.
\end{proof}

\subsection*{Boundary symbols and operators}
We next define the space of boundary symbols of order $m$, class $d$ and 
Hörmander type $(1,\delta)$ by 
\begin{align*}
\mathcal{BM}^{m,d}_{1,\delta}:=\begin{pmatrix}
\mathcal{P}^m_{1,\delta}+\mathcal{G}^{m,d}_{1,\delta} &\mathcal{K}^m_{1,\delta}\\
\mathcal{T}^{m,d}_{1,\delta} & S^m_{1,\delta}
\end{pmatrix}
\end{align*}
It is clear from Theorems \ref{thm_operator_description} and \ref{thm_transmission_normal} that the action of $b\in\mathcal{BM}^{m,d}_{1,\delta}$ in the normal direction defines a 
matrix of operator-valued symbols
\begin{align*}
\op_n(b):=\begin{pmatrix}
\op_n(p)_++\op_n(g) & \op_n(k)\\
\op_n(t) & s
\end{pmatrix}
\end{align*}
We write $B:=\op[\op_nb]$ for the associated  operator. We denote the components of $B$  associated with $p,g,k,t$ and $s$ by $P_+,G,K,T,$ and $S$, respectively. 
It is well-known that these operators form an algebra for Hörmander type $(1,0)$. 
The proof given in \cite{Schrohe2001} extends to the case $(1,\delta)$ with obvious modifications.

\begin{thm}[Composition]
Composition yields  a bilinear and continuous map
\begin{align*}
\mathcal{BM}^{m,d}_{1,\delta}\times\mathcal{BM}^{m',d'}_{1,\delta}\rightarrow\mathcal{BM}^{m+m',\max(m'+d,d')}_{1,\delta},\;\;(b,b')\mapsto b\#b',
\end{align*}
where $\#$ is the Leibniz product of operator-valued symbols,  
given by the property that $\op(\op_nb)\op(\op_nb')=\op (\op_nb\#b')$. Moreover
\begin{align*}
b\#b'=pp'-p_0p_0'+b_0\circ_n b_0'+\mathcal{BM}^{m+m'-(1-\delta),\max(m+d',d)}_{1,\delta}.
\end{align*}
Here the subscript $0$ denotes the restriction to $x_n=0$ and $\circ_n$ denotes the point-wise composition, \cite[Theorem 2.6.1]{Grubb1996}. 
\end{thm}

The well-known mapping properties of Boutet de Monvel operators extend to 
operators of Hörmander type $(1,\delta)$. 
We refer to \cite{Grubb1990} for the proof of the following statement 
(in the case $\delta=0$).

\begin{thm} Let $b\in\mathcal{BM}^{m,d}_{1,\delta}$ and  $s>d+1/p-1$. Then
\begin{align*}
B=\op(\op_n b):H^s_p(\mathbb{R}^n_+)\oplus B^{s-1/p}_p(\mathbb{R}^n_+)\rightarrow H^s_p(\mathbb{R}^n_+)\oplus B^{s-m-1/p}_p(\mathbb{R}^n_+)
\end{align*}
is bounded. The map $b\mapsto B$ is continuous.
\end{thm}

\begin{rem}
The above calculus and the continuity properties naturally extend to the case of 
operators acting on vector bundles over compact manifolds with boundary. 
\end{rem}
\section{The Resolvent}
For the proof of Theorem \ref{thm:main_result}, a suitable description of the resolvent $(A_T-\lambda)^{-1}$ is mandatory.
We explain the key idea of how this description is derived in the simple example, where $A=-\Delta$, $T=\gamma_0$, and $\nu=1$. Here, the benefit is that we can point out the main ideas. However, the majority of abstract arguments can be replaced by explicit computations. \\
In the article \cite{Agmon1962}, Shmuel Agmon proved a priori estimates for solutions of the following boundary value problem with spectral parameter:   
\begin{align}
\label{BVP:example_with_sepectral}
\begin{cases}
(1-\Delta-\lambda)_+u&=f\;\;\text{on}\;\;\mathbb{R}^n_+\\
\gamma_0u&=\phi\;\;\text{on}\;\;\mathbb{R}^{n-1}
\end{cases}.\;\;
\end{align}
Writing $\lambda=\mu^2e^{i\theta}$, we observe that, given a solution $u$ of \eqref{BVP:example_with_sepectral}, the function $\tilde{u}:=u\otimes e_\mu$ with $e_\mu(z)=e^{i\mu z}$ solves the elliptic boundary problem
\begin{align}
\label{BVP:example}
\begin{cases}
(1-\Delta+e^{i(\pi+\theta)}D_z^2)_+\tilde{u}&=\tilde{f}\;\;\text{on}\;\;\mathbb{R}^{n+1}_+\\
\gamma_0\tilde{u}&=\tilde{\phi}\;\;\text{on}\;\;\mathbb{R}^n.
\end{cases}
\end{align}
with $\tilde{f}=f\otimes e_\mu$ and $\tilde{\phi}=\phi\otimes e_\mu$. 
For \eqref{BVP:example}, a priori estimates are well-known, but for our purpose, they are not sufficient. However, the basic idea can be extended to provide a relation between the inverses of \eqref{BVP:example} and \eqref{BVP:example_with_sepectral}. The following three operators are of interest:
\begin{align*}
Q_\theta&:=r^+\mathcal{F}^{-1}(\langle\xi\rangle^2+e^{i(\pi+\theta)}\zeta^2)^{-1}\mathcal{F}e^+,\\
K_\theta&:= r^+\mathcal{F}^{'-1}e^{-\kappa_\theta(\xi',\zeta)x_n}\mathcal{F}^{'},\;\;\text{and}\\
G_\theta&:=-K_\theta\gamma_0Q_\theta.
\end{align*}
Here, $\kappa_\theta(\xi',\zeta)$ is the root of the polynomial $\xi_n\mapsto a_{\theta}(\xi,\zeta):=\langle\xi\rangle^2+e^{i(\pi+\theta)}\zeta^2$, with positive real part. Furthermore, $\mathcal F$ and $\mathcal{F}'$, respectively, denote the Fourier transform with respect to all variables and the tangential variables, respectively.  
The identities $A_\theta Q_\theta=1$, $A_\theta K_\theta=0$, $\gamma_0K_\theta=1$, and $\gamma_0(Q_\theta+G_\theta)=0$ can be verified in a quick calculation. Therefore:
\begin{align}
\label{eq:example_inverse}
\begin{pmatrix}
A_{\theta,+}\\
\gamma_0
\end{pmatrix}^{-1}=
\begin{pmatrix}
Q_{\theta,+}+G_\theta & K_\theta
\end{pmatrix}.
\end{align}
The operators belong to Boutet de Monvel's calculus. We denote the symbols by lower case letters. The solution operators to Problem \eqref{BVP:example} and \eqref{BVP:example_with_sepectral} are related.  In order to reveal this relation, we need the following result. 
\begin{lem}\label{lem:Argom's_trick}
Let $p\in S^m_{1,\delta}(\mathbb{R}^n\times\mathbb{R}^{n+1};E,F)$. Then $p_\mu:=p\vert_{\zeta=\mu}\in S^m_{1,\delta}(\mathbb{R}^n\times\mathbb{R}^{n};E,F)$ and the associated operators are related as follows:
\begin{align}\label{eq:P_to_Pmu}
P(u\otimes e_\mu)=(P_\mu u)\otimes e_\mu.
\end{align}
\end{lem}
\begin{proof} For fixed $\mu$, $p_\mu$ is a symbol in view of the estimate: 
\begin{align*}
c\langle\xi\rangle\leq \langle\xi,\mu\rangle\leq C\langle\xi\rangle,\;\; \text{with}\;\; C=C(\mu).
\end{align*}
The following formal computation can be justified using oscillatory integrals.
\begin{align*}
[P(u\otimes e_\mu)](x,z)&=\int e^{ix\xi+iz\zeta} p(x,\xi,\zeta) [\mathcal{F}u](\xi)\delta(\zeta-\mu)\,d\zeta\dbar\xi\\
&=e^{iz\mu}\int e^{ix\xi} p(x,\xi,\mu) [\mathcal{F}u](\xi)\,\dbar\xi\\
&=[(P_\mu u)\otimes e_\mu](x,z).
\end{align*}
The above computation holds for each point, thus Equation \eqref{eq:P_to_Pmu} holds.
\end{proof}
Now, we verify that the function $u:=(Q_{\theta,\mu,+}+G_{\theta,\mu})f+K_{\theta,\mu}\phi$ solves Problem  \eqref{BVP:example_with_sepectral} for given $f$ and $\phi$:
\begin{align*}
[(A-\lambda)_+u]\otimes e_\mu&=A_{\theta,+}[u\otimes e_\mu]=A_{\theta,+}[((Q_{\theta,\mu,+}+G_{\theta,\mu})f+K_{\theta,\mu}\phi)\otimes e_\mu]\\
&=A_\theta(Q_{\theta,+}+G_\theta)(f\otimes e_\mu)+A_\theta K_\theta(\phi \otimes e_\mu)]\stackrel{\eqref{eq:example_inverse}}{=}f\otimes e_\mu.\\
[\gamma_0u]\otimes e_\mu&=\gamma_0(Q_\theta+G_\theta)(f\otimes e_\mu)+\gamma_0 K_\theta(\phi \otimes e_\mu)]\stackrel{\eqref{eq:example_inverse}}{=}\phi\otimes e_\mu.
\end{align*}
Therefore, the inverse of the parameter-dependent problem can be constructed for the inverse of the associated extended problem. For $\lambda=\mu^2e^{i\theta}$:
\begin{align*}
\begin{pmatrix}
(A-\lambda)_+\\
\gamma_0
\end{pmatrix}^{-1}=
\begin{pmatrix}
Q_{\theta,\mu,+}+G_{\theta,\mu} & K_{\theta,\mu}
\end{pmatrix}.
\end{align*}
What we are especially interested in is the left entry on the right hand side. Here, we observe:
\begin{align*}
(Q_{\theta,\mu,+}+ G_{\theta,\mu})L_p(\mathbb{R}^n_+)\subset \mathcal{D}(A_{\gamma_0}):=\{u\in L_p(\mathbb{R}^n_+): A_+u\in L_p(\mathbb{R}_+^n), \gamma_0u=0\}.
\end{align*}
Therefore, we obtain an explicit formula for the resolvent:
\begin{align*}
(A_{\gamma_0}-\lambda)^{-1}=Q_{\theta,\mu,+}+G_{\theta,\mu},\;\; \text{on the ray}\;\; \lambda=e^{i\theta}\mu^2.
\end{align*} 
The example encourages us to initially solve the extended problem:
\begin{align*}
(A+e^{i(\pi+\theta)}D^2_z)_+\tilde{u}=\tilde{f}\\
T\tilde{u}=\tilde{\phi}.
\end{align*}
In general, no explicit formulas for the inverse of the above problem exist.
We will therefore replace the inverse by a parametrix and analyze the resulting error term. 

According to Equation \eqref{eq:P_to_Pmu}, the restriction $\zeta=\mu$ in Lemma \ref{lem:Argom's_trick} commutes with composition. Therefore, for an elliptic symbol $p$ with parametrix $p^{-\#}$ and remainder $r$ we obtain:
\begin{align*}
P_\mu P^{-\#}_\mu= 1 +R_\mu.
\end{align*}
To estimate the error term, we need to analyze the dependence on the parameters $\theta,\mu$ and thus on $\lambda$ of the operators above. The dependence on $\theta$ for $0<\vartheta\leq |\theta|\leq\pi$ is not essential. In fact, we obtain uniform estimates on operator norms that only depend on $\vartheta$. However, the dependence on $\mu$ is essential and will be discussed next. 
\subsection*{The dependence on the spectral parameter $\mu$}
We consider general Boutet de Monvel symbols which have a covariable $\zeta$ with no space dependence, i.e. they are constant with respect to the variable $z$. By restriction $\zeta=\mu$, we obtain again Boutet de Monvel symbols. The norms of the associated operators depend on the parameter $\mu$.
\begin{thm} \label{thm:spectral_decay} Let $0\leq \delta<1$. 
	\begin{itemize}
		\item[(a)] Let $p\in S^{-m}_{1,\delta}(\mathbb{R}^{n}\times\mathbb{R}^{n+1})$ and $m\geq0$. Then
		\begin{align}
		\|P_\mu\|_{\mathcal{L}(L_p(\mathbb{R}^n))}\leq C|p|_*\langle\mu\rangle^{-m}.
		\end{align}
		\item[(b)] Let $g\in \mathcal{G}^{-m,0}_{1,\delta}(\mathbb{R}^{n-1}\times\mathbb{R}^{n})$ and $m>0$. Then
		\begin{align}
		\|G_\mu\|_{\mathcal{L}(L_p(\mathbb{R}_+^n))}\leq C|g|_*\langle\mu\rangle^{-m}. 
		\end{align}
		\item[(c)] Let $k\in\mathcal{K}_{1,\delta}^{-m}(\mathbb{R}^{n-1}\times\mathbb{R}^{n})$ and $m\geq 0$. Then
		\begin{align}
		\|K_\mu\|_{\mathcal{L}(B^{-1/p}_p(\mathbb{R}^{n-1});L_p(\mathbb{R}_+^n))}\leq C|k|_*\langle\mu\rangle^{-m}. 
		\end{align}
		\item[(d)] Let $t\in\mathcal{T}_{1,\delta}^{-m,0}(\mathbb{R}^{n-1}\times\mathbb{R}^{n})$ and $m\geq 1$. Then
		\begin{align}
		\|T_\mu\|_{\mathcal{L}(L_p(\mathbb{R}_+^n);B^{-1/p}_p(\mathbb{R}^{n-1}))}\leq C|t|_*\langle\mu\rangle^{-m+1}. 
		\end{align}
	\end{itemize}
	Here, $C$ denotes a suitable constant and $|p|_*$, $|g|_*$, $|k|_*$, $|t|_*$ suitable seminorms for $p$, $g$, $k$ and $t$, respectively. 
\end{thm}
Before we turn our attention to the proof, let us draw a conclusion form the above theorem which demonstrates its value.
\begin{cor} Let $m\geq m'\geq 0$. Let $b\in \mathcal{BM}^{m,d}_{1,\delta}(\mathbb{R}^n\times\mathbb{R}^{n+1})$ have a parametrix $b^{-\#}\in\mathcal{BM}^{-m',0}_{1,\delta}(\mathbb{R}^n\times\mathbb{R}^{n+1})$. Then $B_\mu$ is invertible for large $\mu$, and $\|B_\mu^{-1}-B_\mu^{-\#}\|_{\mathcal{L}(L_p(\mathbb{R}_+^n)\oplus B^{-1/p}_P(\mathbb{R}^{n-1}))}\leq C|b|_* \langle\mu\rangle^{-N}$ for all $N\in\mathbb{N}_0$. 
\end{cor}

\begin{proof}
	By assumption $b\#b^{-\#}=1-r$ with $r\in \mathcal{BM}^{-\infty}_{1,\delta}(\mathbb{R}^n\times\mathbb{R}^{n+1})$.  
	As $B_\mu B^{-\#}_\mu=1-R_\mu$, Theorem \ref{thm:spectral_decay} implies that
	$\|R_\mu\|\leq C\langle\mu\rangle^{-N}$ 
	for all $N\in\mathbb{N}_0$. 
	For large $\mu$, the inverse of 
	$1-R_\mu$ is given by a Neumann series. Therefore, $B_\mu$ has a right inverse for large $\mu$:
	$$B_\mu^{-1}=B^{-\#}_\mu+B^{-\#}_\mu\sum_{j\in\mathbb{N}} R_\mu^j.$$ Clearly the second summand is rapidly decreasing in $\mu$. Similarly we obtain a left inverse.
\end{proof}
For the proof of Theorem \ref{thm:spectral_decay} we need the following observation.
Since there is no dependence on the space variable $z$ we can interpret a pseudodifferential operator $P$ with symbol in $S^0_{1,\delta}(\mathbb{R}^{n}\times\mathbb{R}^{n+1})$ as a pseudodifferential operator on the cylinder $\mathbb{R}^n\times\mathbb{S}_L$, where $\mathbb{S}_L$ is the circle with radius $L/2\pi$. Then we obtain:
\begin{lem}\label{lem_uniform_estimates} If $p\in S^0_{1,\delta}(\mathbb{R}^{n}\times\mathbb{R}^{n+1})$, then for all $L>0$ we have
\begin{align*}
P:=\op(p)\in \mathcal{L}(L_p(\mathbb{R}^n\times\mathbb{S}_L))\;\;\text{and}\;\; \|P\|_{\mathcal{L}(L_p(\mathbb{R}^n\times\mathbb{S}_L))}\leq C|p|_*.
\end{align*}
Here $C$ is a constant independent of $L$.
\end{lem}

\begin{proof} We first note that $P$ preserves $L$-periodicity:
\begin{align*}
[Pu](x,z+kL):&=\int e^{i(x-y)\xi+i((z+kL)-w)\zeta}p(x,\xi,\zeta)u(y,w)\,dydw\dbar\xi\dbar\zeta\\
&=\int e^{i(x-y)\xi+i(z-(w-kL))\zeta}p(x,\xi,\zeta)u(y,w)\,dydw\dbar\xi\dbar\zeta\\
&=\int e^{i(x-y)\xi+i(z-\tilde{w})\zeta}p(x,\xi,\zeta)u(y,\tilde{w})\,dyd\tilde{w}\dbar\xi\dbar\zeta\\
&=[Pu](x,z), \quad u\in C^\infty_c(\mathbb R^n\times \mathbb S_L).
\end{align*}
We identify $u\in L_p(\mathbb{R}^n\times\mathbb{S}_L)$ with an $L$-periodic function by letting 
\begin{align*}
u=\sum_{j\in\mathbb{Z}}u_j\;\;\text{with}\;\;u_j(x,z):=u\vert_{\mathbb{R}^n\times[-L/2,L/2]}(x,z-Lj).
\end{align*}
Note that for every $j\in\mathbb{Z}$ we have $u_j\in L_p(\mathbb{R}^n\times\mathbb{R})$ and $\|u_j\|_{L_p(\mathbb{R}^n\times\mathbb{R})}=\|u\|_{L_p(\mathbb{R}^n\times \mathbb{S}_L)}$. 
The integral kernel $k= k(x,z,y,w)$ of the pseudodifferential operator $P$ is given by
\begin{align*}
k(x,z,y,w)=\iint e^{i(x-y)\xi+i(z-w)\zeta}p(x,\xi,\zeta)\,\dbar\xi \dbar \zeta.
\end{align*}
Since $p$ is of order zero, we obtain the estimate
\begin{align*}
|k(x,z,y,w)|\leq C|p|_*(|x-y|^2+|z-w|^2)^{-l/2}
\end{align*}
for all even $l\in\mathbb{N}$  with $l>n$ with a suitable seminorm $|p|_*$ for $p$. 
For $|j|\geq2$, $z\in [-L/2,L/2]$ and $w\in \supp u_j$ we have $|z-w|\geq(j-1)L$, hence
\begin{eqnarray*}
\lefteqn{|k(x,z,y,w)|
\leq C|p|_* (|x-y|^2+ (|j|-1)^2L^2)^{-(n+2)/2}}\\
&\le&C|p|_* ((|j|-1)L)^{-(n+2)}\langle|x-y|/(|j|-1)L\rangle^{-(n+2)}.
\end{eqnarray*}
We write $\chi_j$ for the indicator function of $[-L/2+jL,L/2+jL]$. A quick computation shows that
\begin{align*}
\int \chi_0(z)|k(x,z,y,w)|\chi_{j}(w)\,dwdy&\leq C|p|_*L^{-1}(|j|-1)^{-2}\;\;\text{and}\\
\int \chi_0(z)|k(x,z,y,w)|\chi_{j}(w)\,dzdx&\leq C|p|_*L^{-1}(|j|-1)^{-2}.
\end{align*}
Hence we get $L_p$-estimates by Schur's test. More explicitly:
\begin{eqnarray*}\lefteqn{
\|Pu_{j}\|_{L_p(\mathbb{R}^n\times \mathbb S_L)}=\|\chi_0P\chi_{j}u_{j}\|_{L_p(\mathbb{R}^n\times \mathbb{R})}}\\
&\leq& C|p|_* L^{-1}(|j|-1)^{-2}\|u_{j}\|_{_{L_p(\mathbb{R}^n\times \mathbb{R})}}
=C|p|_* L^{-1}(|j|-1)^{-2}\|u\|_{_{L_p(\mathbb{R}^n\times \mathcal{S}_L)}}
\end{eqnarray*}
In particular the right hand side is summable, and for $L\geq1$ we obtain
\begin{eqnarray*}
\lefteqn{\|Pu\|_{L_p(\mathbb{R}^n\times\mathbb{S}_L)}}\\
&=&\sum_{j\in\{-1,0,1\}}\|Pu_j\|_{L_p(\mathbb{R}^n\times \mathbb S_L)}
+\sum_{|j|\ge2} \|Pu_{j}\|_{L_p(\mathbb{R}^n\times \mathbb S_L)}\\
&\leq& C\Big(3|p|_*\|u\|_{L_p(\mathbb{R}^n\times\mathbb{S}_L)} +2\sum_{j\in\mathbb{N}} j^{-2}|p|_*\|u\|_{L_p(\mathbb{R}^n\times\mathbb{S}_L)}\Big)\\
&\leq& C|p|_*\|u\|_{L_p(\mathbb{R}^n\times\mathbb{S}_L)}
\end{eqnarray*}
We still need to prove that the bound also holds for $L<1$. 
Choose $N\in \mathbb N$ so large that $NL\ge1$, and consider an $L$-periodic function as an
$NL$-periodic function. We have $\|u\|_{L_p(\mathbb{R}^n\times\mathbb{S}_{NL})}
= N^{1/p}\|u\|_{L_p(\mathbb{R}^n\times\mathbb{S}_{L})}$  and hence, by the above
argument,
\begin{eqnarray*}
\lefteqn{\|Pu\|_{L_p(\mathbb{R}^n\times\mathbb{S}_{L})} 
=
N^{-1/p}\|Pu\|_{L_p(\mathbb{R}^n\times\mathbb{S}_{NL})} 
}\\
&\le& C |p|_* N^{-1/p}\|u\|_{L_p(\mathbb{R}^n\times\mathbb{S}_{NL})}
= C|p|_* \|u\|_{L_p(\mathbb{R}^n\times\mathbb{S}_{L})} 
\end{eqnarray*}
 for a constant $C$ independent of $NL$.
\end{proof}

\begin{proof}[Proof of Theorem {\rm\ref{thm:spectral_decay}}]
Let us first assume that $p\in S^0_{1,\delta}(\mathbb{R}^n\times\mathbb{R}^{n+1})$.
We write $e_\mu$ for the $2\pi/\mu$-periodic function $[x\mapsto e^{i\mu x}]$. 
For $u\in L_p(\mathbb R^n)$ we take the $L_p$-norm of both sides of Equation \eqref{eq:P_to_Pmu}:
\begin{eqnarray*}\lefteqn{\|P (u\otimes e_\mu)\|_{L_p(\mathbb{R}^n\times\mathbb{S}_{2\pi/\mu})}}\\
&=&\|[P_\mu u]\otimes e_\mu\|_{L_p(\mathbb{R}^n\times\mathbb{S}_{2\pi/\mu})}=
\|P_\mu u\|_{L_p(\mathbb{R}^n)}\|e_\mu\|_{L_p(\mathbb{S}_{2\pi/\mu})}.
\end{eqnarray*}
Since $P$ is of order zero,  Lemma \ref{lem_uniform_estimates} yields 
\begin{eqnarray*}\lefteqn{\|P_\mu u\|_{L_p(\mathbb{R}^n)}\|e_\mu\|_{L_p(\mathbb{S}_{2\pi/\mu})}}\\
&\leq& C|p|_*\|u\otimes e_\mu\|_{L_p(\mathbb{R}^n\times\mathbb{S}_{2\pi/\mu})}=C|p|_*\| u\|_{L_p(\mathbb{R}^n)}\|e_\mu\|_{L_p(\mathbb{S}_{2\pi/\mu})},
\end{eqnarray*}
and part (a) follows for $m=0$. 
For  $m<0$ we can use what we did so far to reduce to the case 
$p(x,\xi,\mu)=\langle\xi,\mu\rangle^{-m}$. 
But for this symbol the statement is a consequence of the $L_p$-mapping property of 
pseudodifferential operators and the following simple estimates.
\begin{align*}
|D^\alpha_\xi \langle \xi,\mu\rangle^{-m}| 
\leq C_\alpha \langle \xi,\mu\rangle^{-m-|\alpha|}\leq C_\alpha \langle \mu\rangle^{-m}\langle \xi\rangle^{-|\alpha|}.
\end{align*}
Now for part (b). We recall that $\tilde{g}\in \widetilde{\mathcal G}^{m,0}_{1,\delta}(\mathbb{R}^{n-1}\times\mathbb{R}^n)$ satisfies the estimates
\begin{align*}
\|[D_{x_n}^lx_n^{l'}D^{l''}_{y_n}y_n^{l'''}D^\alpha_{\xi'}D_{x'}^\beta\tilde{g}_\mu](x',\xi',x_n,\cdot)\|_{L_1(\mathbb{R}_+)}\leq C|g|_* \langle\xi',\mu\rangle^{m-|\alpha|+\delta|\beta|+l-l'+l''-l'''}\\
\|[D_{x_n}^lx_n^{l'}D^{l''}_{y_n}y_n^{l'''}D^\alpha_{\xi'}D_{x'}^\beta\tilde{g}_\mu](x',\xi',\cdot,y_n)\|_{L_1(\mathbb{R}_+)}\leq C|g|_* \langle\xi',\mu\rangle^{m-|\alpha|+\delta|\beta|+l-l'+l''-l'''}.
\end{align*}
So Schur's test implies that $\|D^\alpha_{\xi'}\op_n\tilde{g}_\mu\|_{\mathcal{L}(L_p(\mathbb{R}_+))}\leq C |g|_*\langle\xi',\mu\rangle^{m-|\alpha|}$. We are interested in the integral kernel
\begin{align*}
K(x',y',\mu)=\int e^{i(x'-y')\xi'}\op_n\tilde{g}_\mu(x',\xi')\,\dbar\xi'=\int L^N\left(e^{i(x'-y')\xi'}-1\right)\op_n\tilde{g}_\mu(x',\xi')\,\dbar\xi'.
\end{align*}
with $N\in\mathbb{N}$ and $L:=\sum_{|\alpha|=1} \frac{(x'-y')^\alpha}{|x'-y'|^2}D_{\xi'}^\alpha$.
We take $N=n-1$ and use the fact that $|e^{it}-1|\le 2|t|^\theta$ for $0<\theta< \min(1,|m|)$, to get
\begin{eqnarray*}\lefteqn{\|K(x',y',\mu)\|_{\mathcal{L}(L_p(\mathbb{R}_+))}}\\
&\leq C|g|_*|x'-y'|^{-n+1+\theta}\int |\xi'|^\theta\langle\xi',\mu\rangle^{-m-n+1}\dbar\xi'
\leq C|g|_*|x'-y'|^{-n+1+\theta}\langle\mu\rangle^{-m+\theta}.
\end{eqnarray*}
Choosing $N=n$ we obtain $\|K(x',y',\mu)\|_{\mathcal{L}(L_p(\mathbb{R}_+))}\leq C |g|_*|x'-y'|^{-n}\langle\mu\rangle^{-m-1}$. 
The first estimate for $\langle\mu\rangle|x'-y'|\leq 1$ and the second for $\langle\mu\rangle|x'-y'|> 1$ imply
\begin{align*}
&\|K(x',\cdot,\mu)\|_{L_1(\mathbb{R}^{n-1};\mathcal{L}(L_p(\mathbb{R}_+)))}\leq C|g|_* \langle\mu\rangle^{-m}\;\;\text{and}\\
&\|K(\cdot,y',\mu)\|_{L_1(\mathbb{R}^{n-1};\mathcal{L}(L_p(\mathbb{R}_+)))}\leq C |g|_*\langle\mu\rangle^{-m}.
\end{align*}
In fact, this follows from the the identities  
\begin{eqnarray*}\lefteqn{\int _{\langle\mu\rangle|x'-y'|\leq 1}|x'-y'|^{-n+1+\theta}\langle\mu\rangle ^\theta dx'}\\
&=&\int _{\langle\mu\rangle|x'-y'|\leq 1}(\langle\mu\rangle|x'-y'|)^{-n+1+\theta}\langle\mu\rangle^{n-1} \,dx'
=\int_{|w|\leq 1}  |w|^{-n+1+\theta}\,dw<\infty\\
\lefteqn{\text{and}}\\
\lefteqn{\int _{\langle\mu\rangle|x'-y'|\geq 1}|x'-y'|^{-n}\langle\mu\rangle^{-1} \,dx'
=\int_{|w|\geq 1}  |w|^{-n}\,dw<\infty.}
\end{eqnarray*}
Hence the assertion follows with Schur's test.
\\ For part (c): We recall the well-known fact that every potential operator $K$ can be written as $r^+P\tilde\gamma_0^*$, where $P$ is a pseudodifferential operator of order $-m-1$ whose symbol-kernel is given by $\tilde{p}=E\tilde{k}$;  $E$ is Seeley's extension operator applied to $x_n$, and  $\tilde\gamma_0^*$ is the adjoint to the evaluation $\tilde\gamma_0: H^s_p(\mathbb R^n)\to B_p^{s-1/p}(\mathbb R^{n-1})$, $s>1/p$. It is clear that $K_\mu=r^+P_\mu\tilde\gamma_0^*$. The map
\begin{align*}
S_{1,0}^{-1}(\mathbb{R}^n\times\mathbb{R}^{n+1})\ni \langle\xi,\zeta\rangle^{-1}\mapsto \langle\xi,\mu\rangle^{-1}\in S_{1,0}^{-1}(\mathbb{R}^n\times\mathbb{R}^{n})
\end{align*}
is uniformly bounded with respect to $\mu$. In view of the continuity of $\tilde\gamma_0^*$ from $B_p^{-1/p}(\mathbb{R}^{n-1})$ to $H^{-1}_p(\mathbb{R}^n)$ we have 
\begin{align*}
\|\op(\langle \xi,\mu\rangle^{-1})\tilde\gamma_0^*\|_{\mathcal{L}(B_p^{-1/p}(\mathbb{R}^{n-1}),L_p(\mathbb{R}^n))}\leq C.
\end{align*}
Define $q=p\#\langle \xi,\zeta\rangle^{1}\in S^{-m}_{1,\delta}(\mathbb{R}^n\times\mathbb{R}^{n+1})$. By part (a)
\begin{align*}
\|Q_\mu\|_{\mathcal{L}(L_p(\mathbb{R}^n))}\leq C|q|_*\langle\mu\rangle^{-m}\leq C|k|_*\langle\mu\rangle^{-m}.
\end{align*}
The estimate for $K_\mu$ follows.\\
For part (d) we use a similar aproch. We write $T=\gamma_0 Pe^+$, where $P$ is a pseudodifferential operator of order $m$ with symbol-kernel $\tilde{p}=E\tilde{t}$. Clearly $T_\mu=\gamma_0 P_\mu e^+$. By the same argument as in part (c) we have
\begin{align*}
\|\gamma_0 \op(\langle \xi,\mu\rangle^{-1})\|_{\mathcal{L}(L_p(\mathbb{R}^n);B_p^{-1/p}(\mathbb{R}^{n-1}))}\leq C.
\end{align*}
Define $q=\langle \xi,\zeta\rangle^{1}\#p\in S^{-m+1}_{1,\delta}(\mathbb{R}^n\times\mathbb{R}^{n+1})$. By part (a)
\begin{align*}
\|Q_\mu\|_{\mathcal{L}(L_p(\mathbb{R}^n))}\leq C|q|_*\langle\mu\rangle^{-m}\leq C|k|_*\langle\mu\rangle^{-m+1}.
\end{align*}
The estimate for $T_\mu$ follows.
\end{proof}

\subsection*{The principal symbol of the degenerate singular Green operator}
We will now apply Agmon's trick to our problem. We introduce the operator
$A_\theta:=A+e^{i\theta}D^2_z$ acting on $\mathbb R^n_+\times \mathbb R$. 
The symbol of  $A_\theta$ is $a_\theta(x,\xi,\zeta) = a(x,\xi)+e^{i\theta}\zeta^2
\in S^2_{1,0}(\mathbb R^n\times \mathbb R^{n+1})$,
where $a(x,\xi) $ is the symbol of $A$.  
Assuming that $a$ is homogeneous of degree 2, there exists a constant $c=c(M,\vartheta)$ such that for all $0<\vartheta\leq|\theta|\leq \pi$ the estimate $|a_\theta(x,\xi,\zeta)|\geq c|\xi,\zeta|^2$ holds. In particular, $A_\theta$ is elliptic. 
After possibly replacing $A$  by $A-c$ for some positive constant $c$ we may and will assume that the Dirichlet problem for $A_\theta$ is invertible. 
In the introduction we already pointed out that the solution operator to the Dirichlet problem is an operator in the Boutet de Monvel calculus, i.e.
\begin{align}
\label{eq:solution_operator_dirichlet_problem}
\begin{pmatrix}
(A_\theta)_+\\
\gamma_0
\end{pmatrix}^{-1}=\begin{pmatrix}
Q_{\theta,+}+G_\theta^D & K_\theta^D
\end{pmatrix}.
\end{align}
We will need the principal symbols of the operators $G_\theta^D$ and $K_\theta^D$
and collect the results to fix some notation.

\begin{rem}\label{rem:dirichlet_problem} 
\begin{itemize}
\item[(a)] For fixed $(x',\xi')$, the restriction to the boundary of the principal symbol of $A_\theta$ is  a polynomial of degree two in $\xi_n$. It therefore 
has two roots, say $\pm i\kappa_\theta^\pm(x',\xi',\zeta)$, with $\re \kappa_\theta^\pm\geq0$.
\item[(b)] We have $\kappa_\theta^\pm\in S^1_{1,0}(\mathbb{R}^{n-1}\times\mathbb{R}^{n})$. 
Both are strongly elliptic, i.e. $\re \kappa_\theta^\pm \geq \omega|\xi',\zeta|$ for 
suitable $\omega>0$. 
\item[(c)] The principal symbol of 
$K_\theta^D\in \mathcal{K}^{0}_{1,0}(\mathbb{R}^{n-1}\times\mathbb{R}^n)$ is $(\kappa_\theta^++i\xi_n)^{-1}$.
\item[(d)] The principal symbol of $G_\theta^D\in\mathcal{G}^{-2,0}_{1,0}(\mathbb{R}^{n-1}\times\mathbb{R}^{n})$ is \\$a_{nn}^{-1}(\kappa^+_\theta+\kappa^-_\theta)^{-1}(\kappa_\theta^++i\xi_n)^{-1}(\kappa_\theta^--i\eta_n)^{-1}$.
\end{itemize}

For details see \cite[Section 2]{GrubbSchrohe2001}
\end{rem}
 For large  $-\lambda=e^{i\theta}\mu^2$ define
\begin{align}
G_\theta^T:=-K_\theta^D(TK_\theta^D)^{-\#} T((A_\theta^{-1})_++G^D_\theta).
\end{align} 
The operator $G_\lambda^T$ defined in \eqref{eq:def_GT} coincides with  $G^T_{\theta,\mu}\mod \mathcal{O}(\langle\lambda\rangle^{-N})$ for all $N\in\mathbb{N}$, as operators in $L_p(\mathbb{R}^n_+)$.
Moreover, let $G^{T,*}_\theta$ be any operator with the same principal symbol as $G^T_\theta$. Then according to Theorem \ref{thm:spectral_decay}, $G^T_\lambda=G^{T,*}_{\theta,\mu}\mod o(\langle\lambda\rangle^{-1})$, as operators on $L_p(\mathbb{R}^n_+)$.  
\begin{lem}
\label{lem:symbol_G}
The operator $G^T_{\theta}$ is a singular Green operator with 
symbol  
$g^T_\theta\in \mathcal{G}^{-2,0}_{1,1/2}(\mathbb{R}^{n-1}\times\mathbb{R}^n)$ and principal symbol 
\begin{align*}
g^T_{\theta(-2)}(x',\xi',\zeta;\xi_n,\eta_n)=s^T_\theta(x',\xi',\zeta)(\kappa_\theta^+(x',\xi',\zeta)+i\xi_n)^{-1}(\kappa_\theta^-(x',\xi',\zeta)-i\eta_n)^{-1}
\end{align*}
for suitable $s^{T}_\theta\in S^{-1}_{1,1/2}\left(\mathbb{R}^{n-1}\times\mathbb{R}^{n}\right)$.
The corresponding symbol-kernel is 
\begin{align*}
\tilde g^T_{\theta(-2)}(x',\xi',\zeta;x_n,y_n)=s^T_\theta(x',\xi',\zeta)e^{-\kappa_\theta^+(x',\xi',\zeta)x_n}e^{-\kappa_\theta^-(x',\xi',\zeta)y_n}.
\end{align*}
\end{lem}

\begin{proof} 
Modulo smoothing operators  $G_\theta^T$ is the composition of the potential operator 
$K_\theta^D$, a parametrix 
$S_\theta^{-\#}$ to the pseudodifferential operator $S_\theta:=TK^D_\theta$ on the boundary, 
multiplication by the function $\varphi_1$  introduced in \eqref{eq:T}
and the trace operator 
$\gamma_1 (Q_{\theta,+}+G^D_\theta)$. 
Note that $Q_{\theta,+}+G^D_\theta$ maps into the kernel of $\gamma_0$ 
so that there is no contribution from $\varphi_0\gamma_0$. 
Hence the principal symbol of $G_\theta^T$ is given by multiplication of the principal symbols of these operators. For the proof of the lemma it is therefore sufficient to combine the following three statements.
\begin{itemize}
\item[(i)] $K^D_\theta = \op k_\theta$ with  $k_\theta\in \mathcal{K}^{0}_{1,0}(\mathbb{R}^{n-1}\times\mathbb{R}^n)$ and  principal symbol
\begin{align*}
k_{\theta(0)}(x',\xi',\zeta,\xi_n)=(\kappa^+_\theta(x',\xi',\zeta)+i\xi_n)^{-1},
\end{align*}
which is Remark \ref{rem:dirichlet_problem}(c).
\item[(ii)]
The symbol $s^{-\#}_\theta\#\varphi_1$ of $S^{-\#}_\theta \varphi_1$ is an element of $S^{-1}_{1,1/2}(\mathbb{R}^{n-1}\times\mathbb{R}^n)$. This is the content of Lemma \ref{lem_S}, below.
\item[(iii)] $\gamma_1 (Q_{\theta,+}+G^D_\theta)=\op t_\theta$ with $t_\theta\in \mathcal{T}^{-1,0}_{1,0}$ and principal symbol
\begin{align*}
t_{\theta(-1)}(x',\xi',\zeta,\xi_n)=-a_n(x')^{-1}(\kappa_\theta^-(x',\xi',\zeta)-i\xi_n)^{-1},
\end{align*}
which follows from Remark \ref{rem:dirichlet_problem} and the composition rules.
\end{itemize}
\end{proof}

\subsection*{The parametrix on the boundary}
We recall a sufficient condition for the existence of a parametrix.

\begin{thm}[Parametrix]
\label{thm_parametrix}
Let $m\geq0$ and $p\in S_{1,0}^m(\mathbb{R}^n\times\mathbb{R}^n)$. Suppose there exists a $0\leq \delta<1$, such that for sufficiently large $|\xi|$ we have the estimates
\begin{align}
\label{eq:hypo_1}
&|p(x,\xi)|\geq c \;\text{ and }
\\
\label{eq:hypo_2}
&|\partial^\beta_x\partial_\xi^\alpha p(x,\xi)p(x,\xi)^{-1}|\leq C\langle \xi\rangle^{-|\alpha|+\delta|\beta|}\;\;\text{for all}\;\;\alpha,\beta\in\mathbb{N}^n_0.
\end{align}
Then there exists a parametrix $p^{-\#}\in S^{0}_{1,\delta}(\mathbb{R}^n\times\mathbb{R}^n)$, i.e.,
\begin{align*}
p^{-\#}\#p=1+ r_1\;\text{and}\;p\#p^{-\#}=1+r_2,
\end{align*}
with $r_1,r_2\in S^{-\infty}(\mathbb{R}^n\times\mathbb{R}^n)$.
\end{thm}

\begin{proof}
See \cite[Chapter 2, Theorem 5.4]{Kumano-Go.1981}.
\end{proof}

\begin{lem}
\label{lem_S} The operator
$S_\theta:=TK_\theta^D$ has a parametrix with symbol $s_\theta^{-\#}$ in $S^{0}_{1,1/2}\left(\mathbb{R}^{n-1}\times\mathbb{R}^n\right)$.
Moreover $s^{-\#}_\theta\#\varphi_1\in S^{-1}_{1,1/2}\left(\mathbb{R}^{n-1}\times\mathbb{R}^n\right)$.
\end{lem}

Before going into the proof let us point out that the difference between the Robin and the degenerate boundary value problem is the order of the operator $S_\theta$ which 
here is zero due to the zeros of $\varphi_1$ and the resulting loss of ellipticity. 
The key observation is that we gain back the loss in order by composing with the multiplication operator $\varphi_1$.

\begin{proof}
We want to show that the symbol of $S_\theta=TK_\theta^D$ satisfies inequalities \eqref{eq:hypo_1} and \eqref{eq:hypo_2}. Write 
\begin{align*}
TK^D_\theta=\varphi_1\gamma_1K^D_\theta+\varphi_0\gamma_0 K^D_\theta
=\varphi_1 \Pi_\theta+\varphi_0,
\end{align*}
where $\Pi_\theta:=\gamma_1K^D_\theta$ is the Dirichlet-to-Neumann operator. 
It is well-known and a consequence of Remark \ref{rem:dirichlet_problem}(c) that its symbol
$\pi_\theta$ is an element of $S^1_{1,0}(\mathbb{R}^{n-1}\times\mathbb{R}^n)$; 
its  principal symbol is $\kappa^+_\theta$.
By Remark \ref{rem:dirichlet_problem}(b) we have $\re\pi_\theta\geq 1$ 
for sufficiently large $|\xi,\zeta|$. Hence, the symbol $s_\theta$
of $S_\theta$ satisfies:
\begin{align}\label{eq_lower_bound}
|s_\theta|\geq |\re(\varphi_1\pi_\theta+\varphi_0)|=\varphi_1\re\pi_\theta+\varphi_0\geq \varphi_1+\varphi_0\geq c>0.
\end{align}
The constant $c$ exists by assumption. We have to verify the estimates
\begin{align*}
|\partial_{x'}^\beta\partial_{\xi'}^\alpha\partial_\zeta^l s_\theta s_\theta^{-1}|\leq \langle\xi',\zeta\rangle^{-|\alpha|-l+|\beta|/2}\;\text{for all}\; \alpha,\beta\in\mathbb{N}^{n-1}_0, l\in\mathbb{N}_0.
\end{align*}
The estimate is trivial for $|\beta|\geq 2$, as
$s_\theta\in S^1_{1,0}(\mathbb{R}^{n-1}\times\mathbb{R}^n)$ and 
$|s_\theta^{-1}|\leq c^{-1}$ by  
Equation \eqref{eq_lower_bound}. 
Equation \eqref{eq_lower_bound}  also shows that 
$(\varphi_1\pi_\theta)^{k/2}s_\theta^{-1}$ is bounded for $k=1,2$. 
The ellipticity of $\pi_\theta$ implies that $|\pi_\theta|^{-k/2}\lesssim\langle\xi',\zeta\rangle^{-k/2}$. We obtain the remaining estimates:
\begin{align*}
|\partial^\alpha_{\xi'}\partial_\zeta^l s_\theta s_\theta^{-1}|
&\equiv|\varphi_1 \partial^\alpha_{\xi'}\partial_\zeta^l\pi_\theta s_\theta^{-1}|
=|\partial^\alpha_{\xi'}\partial_\zeta^l\pi_\theta\pi_\theta^{-1}|
|\varphi_1 \pi_\theta(\varphi_1 \pi_\theta+\varphi_0 )^{-1}|
\lesssim \langle\xi'\rangle^{-|\alpha|-l}\;\\
\intertext{and with the help of the inequality 
$|\partial_{x_j}\varphi_1(t)|^2 \le \|\varphi_1''\|_\infty|\varphi_1(t)|$:}
|\partial_{x_j}\partial^\alpha_{\xi'}\partial_\zeta^l s_\theta s_\theta^{-1}|
&\equiv|\partial_{x_j}\varphi_1 \partial^\alpha_{\xi'}
\partial_\zeta^l\pi_\theta s_\theta^{-1}|
\lesssim\|\varphi_1''\|^{1/2}_\infty|(\varphi_1 \pi_\theta)^{1/2}s_\theta^{-1}||\pi_\theta|^{-1/2}|\partial^\alpha_{\xi'}\partial_\zeta^l\pi_\theta|\\
&\lesssim \langle\xi'\rangle^{1/2-|\alpha|-l}.
\end{align*}
Here $\equiv$ means equality modulo terms that satisfy the estimate. 
According to Theorem \ref{thm_parametrix}, there exists a parametrix to $S_\theta$ with symbol $s_\theta^{-\#}\in S^0_{1,1/2}(\mathbb{R}^{n-1}\times\mathbb{R}^{n})$.
We still need to show that multiplication by $\varphi_1$ reduces the order.
As $\pi_\theta$ is elliptic, there exists a parametrix $\pi_\theta^{-\#}$ such that $\pi_\theta\pi_\theta^{-\#}-1=r'_\theta$ is regularizing, and we find that 
\begin{align*}
\varphi_1 =s_\theta\#\pi_\theta^{-\#}-\varphi_1 \#r'_\theta-\varphi_0 \#\pi_\theta^{-\#}.
\end{align*}
Composition with $\varphi_1 $ or $\varphi_0 $ from the left is just pointwise multiplication. Hence we obtain the improved order of  $s_\theta^{-\#}\#\varphi_1 $ from 
the identities
\begin{align*}
s_\theta^{-\#}\#\varphi_1 
&\equiv s_\theta^{-\#}\#[s_\theta\#\pi_\theta^{-\#}-\varphi_0 \pi_\theta^{-\#}]&\mod  & S^{-\infty}(\mathbb{R}^{n-1}\times\mathbb{R}^n)\text{\;\;and}\\
&\equiv \pi_\theta^{-\#}-s_\theta^{-\#}\#\varphi_0 \pi_\theta^{-\#}&\mod  & S^{-\infty}(\mathbb{R}^{n-1}\times\mathbb{R}^n).
\end{align*}
As $\varphi_0 \pi_\theta^{-\#}, \pi_\theta^{-\#}\in S^{-1}_{1,0}$ and $s_\theta^{-\#} \in S^{0}_{1,1/2}$, this completes the proof. 
\end{proof}

\section{Bounded $H^\infty$-calculus}
In this section we will prove Theorem \ref{thm:main_result}. 
\subsection*{The half space and constant coefficients} 
First, we consider the case where the underlying manifold is the euclidean half-space, the coefficients of the differential operator are constant and only the top order terms are non-zero. In symbols, $X=\mathbb{R}^n_+$, $a_{ij}(x)=a_{ij}\in\mathbb{R}$, $b_j(x)=0$ and $c(x)=0$. According to the last section, the resolvent of $A_T+\nu$ has the following structure:
\begin{align*}
(A_T+\nu-\lambda)^{-1}=Q'_{\theta,\mu,+}+G'_{\theta,\mu}+R(\lambda),
\end{align*}
where $R(\lambda)\in\mathcal{L}(L_p(\mathbb{R}^n_+))$ and $\|R(\lambda)\|=\mathcal{O}(\langle\lambda\rangle^{-1-\varepsilon})$ for some $0<\varepsilon$.
For the proof of Theorem  \ref{thm:main_result} it is sufficient to provide Estimate \eqref{eq:Hinfty_estimate}. According to the equation above, we may estimate the three terms on the right hand side separately. The estimate for the first term is well-known, in fact it is the same as in the non-degenerate case. Any operator whose norm in $\mathcal{L}(L_p(\mathbb{R}^n_+))$ is $\mathcal{O}(\langle\lambda\rangle^{-1-\varepsilon})$, for some $0<\varepsilon$, is integrable along the boundary of $\Sigma_\theta$ and therefore the estimate holds. To provide the estimate for the singular Green part we need the following.
\begin{lem}\label{lem:exp_symbol} Let $\sigma\in S^1_{1,0}(\mathbb{R}^{n-1}\times\mathbb{R}^n)$ and $\re\sigma(x',\xi',\zeta)\geq c|\xi',\zeta|$. Then the map
\begin{align*}
\mathbb{R}_+\ni t\mapsto \exp(-\sigma(x',\xi',\zeta)t)\in S^0_{1,0}(\mathbb{R}^{n-1}\times\mathbb{R}^{n})
\end{align*}
is uniformly bounded. In fact, we have a bound $C=C(|\sigma|_*,c)$ on the seminorms.
\end{lem}
\begin{proof}
Induction over $|\alpha|+|\beta|+l=N$ shows that  $D^\alpha_{\xi'}D^\beta_{x'}D^l_\zeta \exp(-\sigma(x',\xi',\zeta)t)$ is a linear combination over all $k\leq N$, $\alpha_1+\dots+\alpha_k=\alpha$, $\beta_1+\dots+\beta_k=\beta$, and $l_1+\dots+l_k=l$. The terms in the linear combination have the following structure:
\begin{align*}
\left(D^{\alpha_1}_{\xi'}D^{\beta_1}_{x'}D^{l_1}_{\zeta}\sigma(x',\xi',\zeta)\cdots D^{\alpha_k}_{\xi'}D^{\beta_k}_{x'}D^{l_k}_{\zeta}\sigma(x',\xi',\zeta)\right)(-t)^k\exp(-\sigma(x',\xi',\zeta) t).
\end{align*}
Furthermore, the assumption $\sigma\in S^1_{1,0}(\mathbb{R}^{n-1}\times\mathbb{R}^n)$ implies:
\begin{eqnarray*}\lefteqn{
\left|D_{\xi'}^{\alpha_1} D_{x'}^{\beta_1}D_\zeta^{l_1}\sigma(x',\xi',\zeta)\cdots D^{\alpha_n}_{\xi'} D^{\beta_k}_{x'}D_\zeta^{l_k}\sigma(x',\xi',\zeta)\right|}\\
&\leq& 
\prod_{i=1}^k |\sigma|_*|\xi',\zeta|^{1-|\alpha_i|-l_i}=|\sigma|^k_*|\xi',\zeta|^{k-|\alpha|-l}.
\end{eqnarray*}
Moreover, we use the fact that $s^k\exp(-s)$ is bounded on the positive real axis in order to obtain:
\begin{eqnarray*}\lefteqn{
\left|(-t)^k\exp(-\sigma(x',\xi',\zeta) t)\right|}\\
&=& t^k\exp(-\re \sigma(x',\xi',\zeta) t)\leq t^k\exp(-c|\xi',\zeta|t)\leq c^{-k}|\xi',\zeta|^{-k}C.
\end{eqnarray*}
According to the last two estimates, all terms in the linear combination can be estimated by $C|\xi',\zeta|^{-|\alpha|-l}$.
\end{proof}
\begin{lem}\label{lem:est_G}A constant $C=C(|t|_*,M,\vartheta)$ exists such that
\begin{align*}
\left\|\int_{\partial\Sigma_{\theta}}f(\lambda)G'_\lambda\,d\lambda\right\|_{\mathcal{B}(L_p(\mathbb{R}^n_+))}\leq C\|f\|_{L_\infty(\Sigma_\vartheta)}\;\;\text{for all}\;\;f\in H^\infty_0(\Sigma_\vartheta).
\end{align*}
\end{lem}
\begin{proof} The boundary of $\Sigma_\theta$ consists of the two rays $e^{\pm i\theta}\mathbb{R}$, which can be treated separately and analogously. Thus, providing the estimate for the following operator is  sufficient:
\begin{align*}
I^+:=2^{-1}e^{-i{\theta}}\int_{\lambda=e^{i\theta}\mu^2}f(\lambda)G'_\lambda\,d\lambda=\int_0^\infty \mu f(\mu^2 e^{i\theta})G'_{\theta,\mu}\,d\mu.
\end{align*}
For the estimate, we use the explicit description of the symbol-kernel of $G'_{\theta}$ in Lemma \ref{lem:symbol_G}.
Since $s^T_{\theta}\in S^{-1}_{1,1/2}(\mathbb{R}^{n-1}\times\mathbb{R}^n)$, $\zeta s^T_{\theta}(x',\xi',\zeta)\in S^0_{1,1/2}(\mathbb{R}^{n-1}\times\mathbb{R}^n)$. According to Remark \ref{rem:dirichlet_problem}, the roots $\kappa^\pm_{\theta}$ are strongly elliptic and a constant $c=c(M,\vartheta)>0$ exists such that:
\begin{align*}
\re\kappa^\pm_{\theta}(x',\xi',\zeta)\geq 2c|\xi,\zeta|.
\end{align*}
Thus, $\sigma^\pm_{\theta}(x',\xi',\zeta):=\kappa^\pm_{\theta}(x',\xi',\zeta)-c\zeta$ satisfies the assumption of Lemma \ref{lem:exp_symbol} and
the map, below, is uniformly bounded:
\begin{eqnarray*}\lefteqn{\mathbb{R}^2_{++}\ni(x_n,y_n)\mapsto h_{\theta}(x',\xi',\zeta;x_n,y_n)
}\\
&:=&\zeta e^{c\zeta(x_n+y_n)} \tilde{g}'_{\theta}(x',\xi',\zeta;x_n,y_n)\in S^0_{1,1/2}(\mathbb{R}^{n-1}\times\mathbb{R}^n).
\end{eqnarray*}
Now, we analyze the action of $G'_{\theta,\mu}$ in the direction transversal to the boundary. To this end, we define a family of operators that act on $\mathcal{S}(\mathbb{R}^{n-1})$:
\begin{align*}
[G'_{\theta,\mu}(x_n,y_n)v](x'):=\int e^{ix'\xi'}\tilde{g}'_{\theta,\mu}(x',\xi';x_n,y_n)\hat{v}(\xi')\,\dbar\xi'.
\end{align*}
Correspondingly, we define $H_{\theta,\mu}(x_n,y_n)$ from $h_\theta$. Please note that:
\begin{align*}
\mu e^{c\mu(x_n+y_n)}G'_{\theta,\mu}(x_n,y_n)=H_{\theta,\mu}(x_n,y_n).
\end{align*}
Since the seminorms of $h_{\theta}$ are uniformly bounded with respect to $(x_n,y_n)\in\mathbb{R}^2_{++}$, Theorem \ref{thm:spectral_decay} shows that:
\begin{align*}
\|\mu G'_{\theta,\mu}(x_n,y_n)v\|_{L_p(\mathbb{R}^{n-1})}&\leq e^{-c\mu(x_n+y_n)}\|H_{\theta,\mu}v\|_{L_p(\mathbb{R}^{n-1})}\\
&\leq e^{-c\mu(x_n+y_n)}C\|v\|_{L_p(\mathbb{R}^{n-1})}.
\end{align*}
Furthermore, if $u=v\otimes w\in\mathcal{S}(\mathbb{R}^{n-1})\otimes\mathcal{S}(\mathbb{R}_+)$ is a simple tensor, then: 
\begin{align*}
[I^+ u](x',x_n)=\int_0^\infty\int _0^\infty f(\mu^2 e^{i\theta})[\mu G'_{\theta,\mu}(x_n,y_n)v](x')w(y_n)\,dy_nd\mu.
\end{align*}
In order to provide the estimate for $I^+$,
it is sufficient to consider simple tensors because they span a dense subset of $L_p(\mathbb{R}^n_+)$. Therefore:
\begin{align*}
\|I^+ u\|_{L_p(\mathbb{R}^n_+)}&\leq \|f\|_{\infty}\left\|\int_0^\infty\int _0^\infty \|\mu G_{\theta,\mu}(x_n,y_n)v\|_{L_p(\mathbb{R}^{n-1})}|w(y_n)|\,dy_nd\mu \right\|_{L_p(\mathbb{R}_+)}\\
&\hspace*{-3em}\leq C\|f\|_{\infty}\|v\|_{L_p(\mathbb{R}^{n-1})}\left\|\int_0^\infty\int _0^\infty \exp(-c\mu(x_n+y_n))|w(y_n)|\,dy_nd\mu \right\|_{L_p(\mathbb{R}_+)}\\
&\hspace*{-3em}\leq C\|f\|_{\infty}\|v\|_{L_p(\mathbb{R}^{n-1})}\left\|\int_0^\infty \frac{|w(y_n)|}{x_n+y_n}\,dy_n \right\|_{L_p(\mathbb{R}_+)}\\
&\hspace*{-3em}\leq C\|f\|_{\infty}\|v\|_{L_p(\mathbb{R}^{n-1})}\|w\|_{L_p(\mathbb{R}_+)}
= C\|f\|_{\infty}\|u\|_{L_p(\mathbb{R}^n_+)},
\end{align*}
where we used $L_p$-boundedness of the Hilbert transform for the latter inequality. The estimate implies that $I^+\in \mathcal{B}(L_p(\mathbb{R}^n_+))$ and $\|I^+\|\leq C\|f\|_{L_\infty(\Sigma_\vartheta)}$. Here, $C=C(M,|t|_*,\vartheta)$ is the constant in the estimate above.
\end{proof}
We now have proven Theorem \ref{thm:main_result} for differential operator with constant coefficients. 
\begin{rem} The above arguments also provide the result for the case of smooth coefficients. However, in this case the constants also depend on the symbol seminorms of the differential operator.
\end{rem} 
\subsection*{The euclidean half space}\label{sec:non-smooth,half-space}
Now, we treat the situation where $X=\mathbb{R}^n_+$, but the coefficients of the differential operator may non be constant. We assume that $a_{ij}\in C^\tau(\mathbb{R}^n_+)$ for some $\tau>0$ and $b_j,c\in L_\infty(\mathbb{R}^n_+)$.
We use the classical approach of freezing coefficients. We only freeze the coefficients of the differential operator, \emph{not} those of the boundary operator. We use a localization scheme similar to that used by Kunstmann and Weis in \cite{Kunstmann2004}. This provides a family of operators that are small perturbations of an operator with frozen coefficients. We will prove that they allow a bounded $H^\infty$-calculus in a uniform manner. By patching together these operators, we can conclude that $A_T$ itself allows a bounded $H^\infty$-calculus.
We choose a small $r>0$, how small we have to chose $r$ will become clear later on. 
We define the cubes $Q=(-r,r)^n$ and $Q_l:=Q+l$, with $l\in\Gamma:=r(\mathbb{Z}\times\mathbb{N}_0)$. Observe that $\mathbb{R}^n_+\subset\cup_{l\in\Gamma}Q_l$.
We fix a positive function $\psi\in C_c^\infty(Q)$ such that $\gamma_1\psi=0$ and
\begin{align}\label{eq:partition of unity}
\sum_{l\in\Gamma}\psi_l(x)=1\;\;\text{for all}\;\; x\in\mathbb{R}^n_+,\;\;\text{where}\;\; \psi_l(x)=\psi(x-l).
\end{align}
Moreover, we choose a cut-off function $\chi\in C^\infty_c(Q)$ such that $\chi=1$ on $\supp\psi$ and define $\chi_l(x):=\chi(x-l)$. We define $A_l$ as the $L_p$-realization with respect to the boundary operator $T$ of the following differential operator.
\begin{align*}
\mathcal{A}_l=\mathcal{A}^c_l+\mathcal{A}^s_l=\sum_{|\alpha|=2}a_\alpha(l)D^\alpha+\sum_{|\alpha|=2}\chi_l(x)[a_\alpha(x)-a_\alpha(l)]D^\alpha
\end{align*}
Observe that $A_l\psi_l=A_T'\psi_l$, where $A_T'$ denotes the $L_p$-realization of the principal part of $\mathcal{A}$. The major technical difficulty is to show that the each operator in the family $(A_l)_{l\in\Gamma}$ allows a bounded $H^\infty$-calculus, with uniform estimates. More precisely, for suitably chosen $r>0$:
\begin{lem}\label{lem:technical} The operator $A_l$ belongs to $H^\infty(\Sigma_{\theta})$ for all $\theta>0$ and $l\in\Gamma$. Moreover there exists a $C:=C(M,\theta,\|a_\alpha\|_{C^\tau},|t|_*)>0$ such that
\begin{align*}
\|f(A_l)\|_{\mathcal{B}(L_p(\mathbb{R}^n_+))}\leq C\|f\|_\infty\;\;\text{for all}\;\; f\in H^\infty(\Sigma_{\theta})\;\;\text{and}\;\; l\in\Gamma.
\end{align*}
\end{lem}
We can choose $r>0$ such that $A^s_l$ is a small perturbation of $A_l^c+\nu$, in the sense of the following result. Recall that the shift $\nu$ was introduced to ensure the existence of a unique solution to the boundary problem.
\begin{thm}\label{thm:pertubation} Let $E$ be a Banach space with the $UMD$ property, let $A\in \mathcal{S}(E)$ have a bounded $H^\infty(\Sigma_\vartheta)$-calculus, and $0\in\rho(A)$. Suppose that $B$ is a linear operator in $E$ such that $\mathcal{D}(A)\subset\mathcal{D}(B)$ and
\begin{align*}
\|Bu\|_E\leq \varepsilon\|Au\|_E\;\;\text{for all}\;\; u\in\mathcal{D}(A),
\end{align*}
for some $\varepsilon>0$.
Suppose further that $\gamma\in(0,1)$ and a constant $C>0$ exist such that
\begin{align*}
B(\mathcal{D}(A^{1+\gamma}))\subset \mathcal{D}(A^\gamma)\;\;\text{and}\;\; \|A^\gamma Bx\|_E\leq C\|A^{1+\gamma} x\|_{E}\;\;\text{for}\;\;x\in\mathcal{D}(A^{1+\gamma}).
\end{align*}
Then $A+B$ has a bounded $H^\infty(\Sigma_\vartheta)$-calculus in $E$, provided $\varepsilon$ is sufficiently small. Moreover, a constant $C_{A+B}=C_{A+B}(C_A,\varepsilon,C)$ exists such that 
\begin{align*}
\|f(A+B)\|_{\mathcal{B}(E)}\leq C_{A+B}\|f\|_\infty.
\end{align*}
\end{thm}
For the proof we refer to \cite{Denk2004}. To verify the assumptions of the theorem above, we observe:
\begin{lem}\label{lem:small_coefficents} A constant $C>0$ exists such that for $a^s_{l,\alpha}:=\chi_l(a_\alpha-a_\alpha(l))$:
\begin{align*}
\|a_{l,\alpha}^s\|_\infty\leq C\|a_\alpha\|_{C^\tau(\mathbb{R}^n_+)}r^\tau\;\;\text{and}\;\;\|a_{l,\alpha}^s\|_{C^\sigma(\mathbb{R}^n_+)}\leq C\|a_\alpha\|_{C^\tau(\mathbb{R}^n_+)}r^{\tau-\sigma},
\end{align*}
given that $0<\sigma\leq \tau$.
\end{lem} 
\begin{proof}
We recall that $r$ is proportional to the diameter of the cube $Q$. Thus, 
\begin{align*}
\|a_{l,\alpha}^s\|_\infty&\leq \sup \left\{\frac{|a_\alpha(x)-a_{\alpha}(l)|}{|x-l|^\tau}|x-l|^\tau: x\in\supp(\chi_l)\right\}\nonumber\\
&\leq C\|a_\alpha\|_{C^\tau(\mathbb{R}^n_+)} r^\tau.
\end{align*}
By a similar argument, we obtain the second estimate.
\end{proof}
Next, we verify that the lemma above implies the following estimate: 
\begin{align}\label{eq:small_pertubation_1}
\|A^s_lu\|_{L_p(\mathbb{R}^n_+)}\leq Cr^\tau\|(A_l^c+\nu)u\|_{L_p(\mathbb{R}^n_+)}\;\;\text{for all}\;\; u\in H^2_p(\mathbb{R}^n_+)\cap \ker T.
\end{align}
It is well-known that $C^\tau(\mathbb{R}^n_+)\hookrightarrow\mathcal{B}(H^s_{p}(\mathbb{R}^n_+))$ as a multiplication operator for $0\leq s\leq \tau$. Therefore, with $s=0$ we obtain:
\begin{align*}
\|A^s_lu\|_{L_p(\mathbb{R}^n_+)}\leq \sum_{1\leq i,j\leq n}\|a_{l,ij}^s\|_{C(\mathbb{R}^n_+)}\|u\|_{H^2_p(\mathbb{R}^n_+)}\leq Cr^\tau\|u\|_{H^2_p(\mathbb{R}^n_+)}. 
\end{align*}
Furthermore,  on $H^2_p(\mathbb{R}^n_+)\cap \ker T$, the norm $\|(A^c_l+\nu)\cdot\|_{L_p(\mathbb{R}^n_+)}$ and the $H^2_p(\mathbb{R}^n_+)$ norm are equivalent because $(A^c_l+\nu)$ is invertible. Hence, Equation \eqref{eq:small_pertubation_1} holds.
Now, we compute the domain of $(A^c_l+\nu)^\gamma$ for $2\gamma<\min\{1/p,\tau\}$. According to Theorem \cite[Theorem 1.15.2]{Triebel1978}, the domain is:
\begin{align*}
\mathcal{D}((A^c_l+\nu)^\gamma)=[L_p(\mathbb{R}^n_+),H^2_p(\mathbb{R}^n_+)\cap \ker T]_\gamma.
\end{align*}
We write $\dot{H}^2_{p}(\mathbb{R}^n_+)$ for the closure of $C^\infty_c(\mathbb{R}^n_+)$ in $H^2_{p}(\mathbb{R}^n_+)$.
By interpolation, the embedding $\dot{H}^2_{p}(\mathbb{R}^n_+)\hookrightarrow H^2_p(\mathbb{R}^n_+)\cap \ker T\hookrightarrow H^2_p(\mathbb{R}^n_+)$ implies:
\begin{align*}
\dot{H}^{2\gamma}_{p}(\mathbb{R}^n_+)\hookrightarrow[L_p(\mathbb{R}^n_+),H^2_p(\mathbb{R}^n_+)\cap \ker T]_\gamma\hookrightarrow H^{2\gamma}(\mathbb{R}^n_+).
\end{align*}
As $H^{2\gamma}_p(\mathbb{R}^n_+)=\dot{H}^{2\gamma}_{p}(\mathbb{R}^n_+)$ for $2\gamma<1/p$, we conclude that $\mathcal{D}((A^c_l+\nu)^\gamma)=H^{2\gamma}_p(\mathbb{R}^n_+)$. Furthermore, the operator $(A^c_l+\nu)^\gamma$ is invertible. Thus, $\|(A^c_l+\nu)^\gamma\cdot\|_{L_p(\mathbb{R}^n_+)}$ and $\|\cdot\|_{H^{2\gamma}_p(\mathbb{R}^n_+)}$ are equivalent norms on $\mathcal{D}((A^c_l+\nu)^\gamma)$. We make use of Lemma \ref{lem:small_coefficents} and the embedding $C^\sigma(\mathbb{R}^n_+)\hookrightarrow\mathcal{B}(H^s_{p}(\mathbb{R}^n_+))$ to obtain the following estimate:
\begin{align*}
\|(A^c_l+\nu)^\gamma A^s_lu\|_{L_p(\mathbb{R}^n_+)}\leq C\|A^s_lu\|_{H^{2\gamma}_p(\mathbb{R}^n_+)}\leq Cr^{\tau-2\gamma}\|u\|_{H^{2+2\gamma}_p(\mathbb{R}^n_+)}.
\end{align*}
We can further estimate the right hand side with \cite[p.~70]{Krietenstein2019}:
\begin{align*}
\|u\|_{H^{2+2\gamma}_p(\mathbb{R}^n_+)}\leq C\|(\nu+A^c_l)u\|_{H^{2\gamma}_p(\mathbb{R}^n_+)}\leq \|(\nu+A^c_l)^{1+\gamma}u\|_{L_p(\mathbb{R}^n_+)}.
\end{align*}
In sum, the following estimate holds for all $u\in \mathcal{D}((\nu+A^c_l)^{1+\gamma})$:
\begin{align}\label{eq:small_pertubation_2}
\|(\nu+A^c_l)^\gamma A^s_lu\|_{L_p(\mathbb{R}^n_+)}\leq Cr^{\tau-2\gamma}\|(\nu+A^c_l)^{1+\gamma}u\|_{L_p(\mathbb{R}^n_+)}.
\end{align}
The constants in Equation \eqref{eq:small_pertubation_1} and \eqref{eq:small_pertubation_2} are independent of $l$ and $r$. Therefore, we can choose $r$ such that Theorem \ref{thm:pertubation} applies to $\nu+A^c_l+A^s_l$ and thus Lemma \ref{lem:technical} holds.\\
Now we describe the localization scheme. We define $\mathbb{H}^s_p(\mathbb{R}^n_+):=l_p(\Gamma,H^s_p(\mathbb{R}^n_+))$ and we write $\mathbb{L}^s_p(\mathbb{R}^n_+)$ if $s=0$. We introduce the localization operator $L$ and the patching operator $P$ with the help of partition of unity \eqref{eq:partition of unity}:
\begin{align*}
& L:L_p(\mathbb{R}^n_+)\rightarrow\mathbb{L}_p(\mathbb{R}^n_+),\;\;  u\mapsto (\psi_lu)_{l\in\Gamma}.\\
& P:\mathbb{L}_p(\mathbb{R}^n_+)\rightarrow L_p(\mathbb{R}^n_+),\;\;  (u_l)_{l\in\Gamma}\mapsto \sum_{l\in\Gamma}\chi_lu_l.
\end{align*}
We also define the operator $\mathbb{T}:\mathbb{H}^2_p(\mathbb{R}^n_+)\rightarrow l_p(\Gamma;B^{1-1/p}_p(\mathbb{R}^n_+))$, $(u_l)_{l\in\Gamma}\rightarrow (Tu_l)_{l\in\Gamma}$.
We collect some properties of these operators, which follow directly form the definitions:
\begin{lem} Let $L$, $P$ and $\mathbb{T}$ be as above. Then 
\begin{enumerate}
\item $L\in\mathcal{B}(H^s_p(\mathbb{R}^n_+);\mathbb{H}^s_p(\mathbb{R}^n_+))$
\item $P\in\mathcal{B}(\mathbb{H}^s_p(\mathbb{R}^n_+);H^s_p(\mathbb{R}^n_+))$
\item $PL=1$
\item $L:H^2_p(\mathbb{R}^n_+)\cap\ker T\rightarrow \mathbb{H}^2_p\cap \ker \mathbb{T}$
\item $P:\mathbb{H}^2_p\cap \ker \mathbb{T}\rightarrow H^2_p(\mathbb{R}^n_+)\cap\ker T$
\end{enumerate}
\end{lem}
We write $A_{lk}:=\delta_{lk}A_l$, with domain $\mathcal{D}(A_{lk})=H^2_p(\mathbb{R}^n_+)\cap \ker T$. We define
\begin{align}
\label{eq:def_Fat_A}
\mathbb{A}:\mathcal{D}(\mathbb{A}):=\mathbb{H}^2_p(\mathbb{R}^n_+)\cap \ker\mathbb{T} \subset\mathbb{L}_p(\mathbb{R}^n_+)\rightarrow \mathbb{L}_p(\mathbb{R}^n_+),\;\;(u_k)_{k\in\Gamma}\mapsto \left(\sum_{k\in\Gamma} A_{lk}u_k\right)_{l\in\Gamma}.
\end{align}
Similar we define $\mathbb{B}$ and $\mathbb{D}$ for the following families of operators.
\begin{align*}
B_{lk}:=\delta_{lk}A_{low}+[\psi_l,A]\psi_k\;\;\text{and}\;\;D_{lk}= \delta_{lk}A_{low}+\psi_l[A_k+A_{low},\psi_k].
\end{align*}
Here $A_{low}$ denotes the $L_p$-realisation with respect to the boundary operator $T$ of $\mathcal{A}-\mathcal{A}'$.
All sums in \eqref{eq:def_Fat_A} are finite. In fact, we have a symmetric relation $l\bowtie k:\Leftrightarrow \supp \psi_l\cap\supp\psi_k\neq\emptyset$ on $\Gamma$. The definition of $\psi_l$ implies that for fixed $l\in\Gamma$ the set $\Gamma_l:=\{k\in\Gamma:k\bowtie l\}$ is finite. Obviously $B_{lk}=0$ and $D_{lk}=0$ if $k\neq \Gamma_l$.
The operators above are defined such that they satisfy the following relations.
\begin{align}
\label{eq:relation_B}
LA&=(\mathbb{A+B})L\;\;\text{on}\;\; \mathcal{D}(A)\;\;\text{and}\\
\label{eq:relation_D}
AP&=P(\mathbb{A+D})\;\;\text{on}\;\; \mathcal{D}(\mathbb{A}).
\end{align}
For suitably chosen $r>0$ we obtain: 
\begin{lem}\label{lem:diagonal}
The operator $\mathbb{A}$ belongs to $H^\infty(\Sigma_\theta)$ for each $\theta>0$.
\end{lem}
\begin{proof}
We fix $\theta>0$ and choose $r>0$ such that Lemma \ref{lem:technical} applies. In particular, $\Sigma_\theta\subset \rho (A_l)$ for all $l\in\Gamma$ with uniform bounds on the inverse. Therefore, the inverse of $\lambda-\mathbb{A}$ exists and is given by $(\lambda-\mathbb{A})^{-1}(u_l)_{l\in\Gamma}=((\lambda-A_l)^{-1}u_l)_{l\in\Gamma}$. For each $l\in\Gamma$ we have a bounded operator
\begin{align*}
\hat{l}:\mathbb{L}_p(\mathbb{R}^n_+)\rightarrow L_p(\mathbb{R}^n_+),\;\;(u_k)_{k\in\Gamma}\rightarrow u_l.
\end{align*}
Let $C$ be as in Lemma \ref{lem:technical} and $f\in H_*^\infty(\Sigma_\theta)$. Then
\begin{align*}
\|f(\mathbb{A})(u_k)_{k\in\Gamma}\|^p_{\mathbb{L}_p(\mathbb{R}^n_+)}&=\sum_{l\in\Gamma}\left\|\hat{l}\int_{\partial\Sigma_\theta}f(\lambda)(\lambda+\mathbb{A})^{-1}(u_k)_{k\in\Gamma}\,d\lambda\right\|^p_{L_p(\mathbb{R}^n_+)}\\
&=\sum_{l\in\Gamma}\left\|\int_{\partial\Sigma_\theta}f(\lambda)(\lambda-A_l)^{-1}u_l\,d\lambda\right\|^p_{L_p(\mathbb{R}^n_+)}\\
&\leq \sum_{l\in\Gamma} C^p\|f\|^p_\infty\|u_l\|_{L_p(\mathbb{R}^n_+)}=C^p\|f\|^p_\infty\|u\|^p_{\mathbb{L}_p(\mathbb{R}^n_+)}.
\end{align*}
This estimate is sufficient to see that $\mathbb{A}\in H^\infty(\Sigma_\theta)$.
\end{proof}
Next, we observe that both $\mathbb{B}$ and $\mathbb{D}$ are lower order perturbations of $\mathbb{A}$ in the sense of the following well-known perturbations theorem going back to Amann. For a proof we refer to \cite[Proposition 13.1]{Kunstmann2004}.
\begin{thm}\label{thm:lower_order_pertubation} Let $A\in\mathcal{S}(E)$ have a bounded $H^\infty(\Sigma_\theta)$-calculus in $E$ and assume $0\in\rho(A)$. Let $\gamma\in(0,1)$ and suppose that $B$ is a linear operator in $E$ satisfying $\mathcal{D}(B)\supset\mathcal{D}(A)$, and
\begin{align*}
\|Bu\|_E\leq C\|A^{1-\gamma}u\|_E\;\;\text{for all}\;\; u\in\mathcal{D}(A),
\end{align*}
where $C>0$. Then $\nu+A+B$ has a bounded $H^\infty(\Sigma_\theta)$-calculus in $E$ for $\nu\geq 0$ sufficiently large.
\end{thm}
In particular, for suitably chosen $r>0$ we obtain:

\begin{lem}\label{lem:band} For each $\theta>0$ a constant $\nu\geq0$ exists such that both $\nu+\mathbb{A+B}$ and $\nu+\mathbb{A+D}$ belong to $H^\infty(\Sigma_\theta)$. 
\end{lem}
\begin{proof}
We can assume that $0\in\rho(\mathbb{A})$, otherwise we consider $\nu+\mathbb{A}$. Thus, $\mathbb{A}^{(1-\gamma)}$ is invertible and $\|\cdot\|_{\mathcal{D}(\mathbb{A}^{1-\gamma})}$ is equivalent to $\|\mathbb{A}^{1-\gamma}\cdot\|_{\mathbb{L}_p(\mathbb{R}^n_+)}$. According to Lemma \ref{lem:diagonal}, the operator $\mathbb{A}$ belongs to $H^\infty(\Sigma_\theta)$ and therefore has bounded imaginary powers. According to \cite[Theorem 1.15.2]{Triebel1978}, the domain of $\mathbb{A}^{1-\gamma}$ is given by complex interpolation.
\begin{align*}
\mathcal{D}(\mathbb{A}^{1-\gamma})=[\mathbb{L}_p(\mathbb{R}^n_+),\mathcal{D}(\mathbb{A})]_{1-\gamma}\hookrightarrow[\mathbb{L}_p(\mathbb{R}^n_+),\mathbb{H}^2_p(\mathbb{R}^n_+))]_{1-\gamma}=\mathbb{H}^{2-2\gamma}_p(\mathbb{R}^n_+))
\end{align*}
We can focus on $\mathbb{B}$, because the arguments for $\mathbb{D}$ are the same. A closer look on the definition of $B_{lk}$ reveals that it is a first order differential operator. In particular, for each $\gamma<1/2$ we have the standard estimate:
\begin{align}\label{eq:estBkl}
\|B_{lk}u\|_{L_p(\mathbb{R}^n_+)} \leq C\|u\|_{H^1_p(\mathbb{R}^n_+)}\leq C\|u\|_{H^{2-2\gamma}_p(\mathbb{R}^n_+)}.
\end{align}
Note that the constant $C>0$ only depends on the $L_\infty$-norm of the coefficients and thus can be chosen independent of $k$ and $l$.
We write $N:=\sup _{l\in\Gamma} \#\{k\in\Gamma :k\bowtie l\}$. Then by estimate \eqref{eq:estBkl} 
\begin{align*}
\|\mathbb{B}(u_k)_{k\in\Gamma}\|^p_{\mathbb{L}_p(\Gamma)}&=\sum_{l\in\Gamma}\left\|\sum_{k\,\bowtie\, l} B_{lk}u_k\right\|^p_{L_p(\mathbb{R}^n_+)}
\leq \sum_{l\in\Gamma}\left(\sum_{k\,\bowtie\, l} C\left\|u_k\right\|_{H^{2-2\gamma}_p(\mathbb{R}^n_+)}\right)^p\\
&\leq \sum_{l\in\Gamma} C^pN^p\sup_{k\,\bowtie\,l}\left\|u_k\right\|_{H^{2-2\gamma}_p(\mathbb{R}^n_+)}^p
\leq C^pN^p\sum_{l\in\Gamma} \sum_{k\,\bowtie\,l}\left\|u_k\right\|_{H^{2-2\gamma}_p(\mathbb{R}^n_+)}^p\\
&\leq  C^pN^{p+1}\sum_{l\in\Gamma}\left\|u_l\right\|_{H^{2-2\gamma}_p(\mathbb{R}^n_+)}^p=C^pN^{p+1}\|(u_l)_{l\in\Gamma}\|_{\mathbb{H}^{2-2\gamma}(\mathbb{R}^n_+)}\\
&\leq C^pN^{p+1}\|\mathbb{A}^{1-\gamma}(u_l)_{l\in\Gamma}\|^p_{\mathbb{L}_p(\mathbb{R}^n_+)}.
\end{align*}
In the fourth inequality we used the symmetry of the relation $\bowtie$ to change the order of summation. We finish the proof by the application of Theorem \ref{thm:lower_order_pertubation} to $\nu+\mathbb{A+B}$.
\end{proof}
Now, we can prove Theorem \ref{thm:main_result} for the case $X=\mathbb{R}_+^n$.
\begin{proof}
For given $\theta>0$ we choose $\nu,r>0$ such that Lemma \ref{lem:band} applies. For each $\lambda\in\Sigma_\theta$ the operator $\lambda-(\nu+A_T)$ is invertible with left inverse $P(\lambda-(\nu+\mathbb{A+B}))^{-1}L$ and right inverse $P(\lambda-(\nu+\mathbb{A+D}))^{-1}L$. For all $f\in H^\infty_*(\Sigma_\theta)$ we have
\begin{align*}
\|f(\nu+A_T)\|\leq \|P\|\|f(\nu+\mathbb{A+B})\|\|L\|\leq C\|f\|_\infty.
\end{align*}
Therefore, $\nu+A_T$ allows a bounded $H^\infty(\Sigma_\theta)$-calculus.
\end{proof}
\subsection*{Manifolds}
Now, let $(X,g)$ be a manifold with boundary and bounded geometry as in \cite{Grosse2013}, see also \cite{Schick2001}. We choose an atlas of Fermi coordinates $\kappa_l:U_l\subset X\rightarrow V_l\subset \overline{\mathbb{R}}^n_+$ with index set $\Gamma$ such that $\sup_{l\in\Gamma}|\{k\in\Gamma:U_k\cap U_l\neq \emptyset\}|=:N<\infty$. We also choose a subordinate partition of unity $(\psi_l)_{l\in\Gamma}$ such that $\partial_\nu\psi_l=0$ for all $l\in\Gamma$. Here, $\nu$ denotes an outward unit normal vector field on $\partial X$. For each $\psi_l$, we choose positive functions $\chi_l',\chi_l\in C^\infty_c(U_l)$ such that $\chi_l=1$ on $\supp \psi_l$ and $\chi_l'=1$ on $\supp\chi_l$. We denote $\chi_{l,*}=\kappa_{l,*}\chi_l\in C_c^\infty(V_l)\subset C_c^\infty(\overline{\mathbb{R}}_+^n)$. Similarly, we define $\chi'_{l,*}$. Moreover, we write $\tilde{\kappa}_l(x'):=\kappa_l(x',0)$ for the induced chart on the boundary. Let $\mathcal{A}$ be a sufficiently regular $M$-elliptic second order differential operator on $X$ as in \eqref{eq:A} and $T$ be a boundary operator as in \eqref{eq:T}. For each $l\in\Gamma$, we define the following operators:  
\begin{align*}
\mathcal{A}_l:=-\Delta (1-\chi'_{l,*})+\kappa_{l,*}\mathcal{A}\kappa_l^*\chi'_{l,*}&&\text{and}&& T_l:=\gamma_0 (1-\chi'_{l,*})+\tilde{\kappa}_{l,*}T\kappa_l^*\chi'_{l,*}.
\end{align*}
Then $\mathcal{A}_l$ is an $M$-elliptic second order differential operator on euclidean space with sufficiently regular coefficients. Moreover, the norms of the coefficients of the local representations of $\mathcal{A}$ are bounded by $M$. Therefore, the norms of the coefficients of $\mathcal{A}_l$ are uniformly bounded with respect to $l\in\Gamma$ and so are the seminorms $|t_l|_*$. We define: 
\begin{align*}
A_l:\mathcal{D}(A_l):=\{u\in H^2_p(\mathbb{R}^n_+):T_lu=0\}\rightarrow L_p(\mathbb{R}^n_+),\;u\mapsto r^+\mathcal{A}_le^+u.
\end{align*}
Each operator $A_l$ satisfies the assumptions in the last subsection. Therefore, we can apply Theorem \ref{thm:main_result} to $A_l$, which implies that Lemma \ref{lem:technical} continues to hold.
We define the localization operator and the patching operator by:
\begin{align*}
L&:L_p(X)\rightarrow \mathbb{L}_p(\mathbb{R}^n_+),\;\; u\mapsto (\kappa_{l,*}\psi_lu)_{l\in\mathbb{N}}.\\
P&:\mathbb{L}_p(\mathbb{R}^n_+)\rightarrow L_p(X),\;\; (u_l)_{l\in\mathbb{N}}\rightarrow \sum_{l\in\mathcal{I}} \kappa_{l}^*\chi_{l,*} u_l.
\end{align*}
By definition, $u$ belongs to $H^s_p(X)$ if and only if $Lu$ belongs to $\mathbb{H}^s_p(\mathbb{R}^n_+)$. Moreover, the norms of $u$ and $Lu$ coincide. Therefore:
\begin{itemize}
\item $L\in\mathcal{B}(H^s_p(X);\mathbb{H}^s_p(\mathbb{R}^n_+))$,
\item $P\in\mathcal{B}(\mathbb{H}^s_p(\mathbb{R}^n_+);H^s_p(X))$, and
\item $PL=1$.
\end{itemize}

\begin{rem} Spaces on manifolds with boundary and bounded geometry:  \begin{itemize}
\item[(a)] It is natural to define $H^s_p(X)$ as $r^+H^s_p(\hat{X})$. Here $r^+$ is the restriction in the sense of distributions, $H^s_p(\hat{X})=(I-\Delta_g)^{-s/2}L_p(\hat{X})$ and $\hat{X}$ is a manifold with bounded geometry which contains $X$. For the existence of $\hat{X}$ we refer to \cite{Ammann2018}. 
The operator $(I-\Delta_g)^{-s/2}$ is well defined for all $s\in\mathbb{R}$, due to the result of Strichartz in \cite{Strichartz1983}. Since the restriction can be treated analogously to the euclidean or compact case, we may only consider $H^s_p(\hat{X})$. 
Let $L$ be defined as above with respect to an atlas of normal coordinates. Then $\|L\cdot\|_{\mathbb{H}_p^s(\mathbb{R}^n)}$ and $\|(I-\Delta_g)^s\cdot\|_{L_p(\hat{X})}$ are equivalent norms; this result is due to H. Triebel, see \cite[Theorem 7.4.5]{Triebel1992}. In \cite{Grosse2013}, it was observed that an atlas of Fermi coordinates also gives rise to an equivalent norm.
\item[(b)] The interpolation results for $H^s_p(\mathbb{R}^n)$ extend to $H^s_p(\hat{X})$. This follows from two facts. First $H^s_p(\hat{X})$ is a retract of $\mathbb{H}^s_p(\mathbb{R}^n)$. Second $\mathbb{H}^s_p(\mathbb{R}^n)$ is the space of p-summable sequences with values in $H^s_p(\mathbb{R}^n)$.
\item[(c)] We may define Besov-spaces via real interpolation or via the localization operator $L$. According to part (b) both definitions coincide. The trace theorem holds on manifolds with boundary and bounded geometry, see \cite{Grosse2013} for the details.
\end{itemize} 
\end{rem}

Furthermore, we define $\mathbb{T}:\mathbb{H}^2_p(\mathbb{R}^n_+)\rightarrow \mathbb{B}^{1-1/p}_p(\mathbb{R}^{n-1})$, $(u_l)_{l\in\mathcal{I}}\mapsto (T_lu_l)_{l\in\mathcal{I}}$. Using the fact $\partial_\nu\psi_l=0$ for all $l\in\Gamma$ we obtain: The localization operator maps the kernel of $T$ to the kernel of $\mathbb{T}$ and the patching operator maps the kernel of $\mathbb{T}$ into the kernel of $T$.
We define $\mathcal{D}(\mathbb{A}):=\mathbb{H}^2_p(\mathbb{R}^n_+)\cap\ker \mathbb{T}$. Note that $(u_l)_{l\in\Gamma}\in \mathcal{D}(\mathbb{A})$ implies that $u_l\in \mathcal{D}(A_l)$ for all $l\in\Gamma$. Therefore, the following definition is reasonable:
\begin{align*}
\mathbb{A}:\mathcal{D}(\mathbb{A}):=\mathbb{H}^2_p(\mathbb{R}^n_+)\cap\ker \mathbb{T}\subset \mathbb{L}_p(\mathbb{R}^n_+)\rightarrow\mathbb{L}_p(\mathbb{R}^n_+),\,(u_l)_{l\in\Gamma}\mapsto (A_lu_l)_{l\in\Gamma}.
\end{align*}
Lemma \ref{lem:diagonal} continues to hold as it only relies on Lemma \ref{lem:technical}.
We define $\mathbb{B},\mathbb{D}:\mathbb{H}^2_p(\mathbb{R}^n_+)\subset\mathbb{L}_p(\mathbb{R}^n_+)\rightarrow \mathbb{L}_p(\mathbb{R}^n_+)$ as infinite matrices with entries:
\begin{align*}
B_{lk}:=\kappa_{l_*}[\psi_l,A_T]\chi'_{k,*}\kappa_{k}^*&&\text{resp.}&& D_{lk}:=\kappa_{l,*}\psi_l\kappa_k^*[A_k,\chi_{k,*}].
\end{align*}
Again, the definition is motivated by the Relations \eqref{eq:relation_B} and \eqref{eq:relation_D}. The operators $\mathbb{A},\mathbb{B},\mathbb{D}, L$ and $P$ have the same properties as those on the euclidean space. Therefore, the proof of Theorem \ref{thm:main_result} carries over.

\section{The Porous Medium Equation}\label{sec:porus_medium_equation}

In this section, we illustrate the applicability of the theory developed so far to nonlinear parabolic partial differential equations. A prominent example for this type of equations is the porous medium equation \eqref{eq:PME_1}.   
It arises for instance in the description of the gas flow through a porous medium. 
As pointed out, we consider the case where the initial value $v_0\in H^2_p(X)$ satisfies $v_0\geq c$ for some $c>0$ and the boundary value is independent of time and compatible with the initial value, i.e., $\phi=Tv_0$.
Under this assumption, Theorem \ref{thm_PME} provides the short time existence of a solution.

The proof, below, is inspired by \cite{Schrohe2016}.
We define $u:=v-v_0$ and consider the following equivalent parabolic problem:
\begin{align}\label{eq:PME_2}
\begin{cases}
\dot{u}-\Delta_g (u+v_0)^m=0\\
Tu=0\\
u\vert_{t=0}=0.
\end{cases}
\end{align}
A quick computation shows that $v$ solves \eqref{eq:PME_1} if and only if $u$ solves \eqref{eq:PME_2}. Therefore, we focus on Problem \eqref{eq:PME_2} which we rewrite as an abstract parabolic problem. To this end, we need the following identity which can easily be verified in local coordinates:
\begin{align*}
\Delta_g(u+v_0)^m=& m(u+v_0)^{m-1}\Delta_gu\\&+m(m-1)(u+v_0)^{m-2}|\nabla (u+v_0)|^2_{g}+m((u+v_0))^{m-1}\Delta_gv_0.
\end{align*}
The first term on the right hand side is the highest order term. Therefore, we define $A(u):=-m(u+v_0)^{m-1}\Delta_{g,T}$ and:
\begin{align*}
f(u):=m(m-1)(u+v_0)^{m-2}|\nabla (u+v_0)|^2_{g}+m((u+v_0))^{m-1}\Delta_{g,T}v_0.
\end{align*}
According to the definitions above, Problem \eqref{eq:PME_2} is the abstract parabolic problem:
\begin{align}\label{eq:abstract_parabolic_problem}
\dot{u}+A_T(u)u=f(u);\;\;u\vert_{t=0}=0.
\end{align}
In the following, we verify that the theorem, below, can be applied to \eqref{eq:abstract_parabolic_problem}. 

\begin{thm}[Cl\'ement \& Li, \cite{Clement1994}]\label{thm:Clement/Li} Given an equation in $L_q(0,T;E_0)$:
\begin{align}\label{eq:AQLPP}
\dot{u}(t)+A(u(t))u(t)=f(t,u(t))\;\;\text{and}\;\; u(t_0)=u_0,
\end{align}
for some $1<q<\infty$, for some finite $T$, and $\mathcal{D}(A(u(t)))=E_1$. 
We assume that $A(u_0)$ has maximal regularity and a neighborhood $U$ of $u_0$ exists in $E_{q}=[E_1,E_0]_{1/q,q}$ such that for all $u,u'\in U$:
\begin{itemize}
\item[(CL1)]\label{eq:cl1} $\|A(u)-A(u')\|_{\mathcal{B}(E_1;E_0)}\leq C\|u-u'\|_{E_q}$.
\item[(CL2)]\label{eq:cl2} $\|f(t,u)-f(t',u')\|_{E_0}\leq C(\|u-u'\|_{E_q}+|t-t'|)$.
\end{itemize}
Then, a $\tau>0$ exists such that the Equation \eqref{eq:AQLPP} has a unique solution in: \begin{align*}
L_q(0,\tau;E_1)\cap H^1_q(0,\tau;E_0).
\end{align*}
\end{thm}
To verify the assumptions, we define $E_0=L_p(X)$ and $E_1=H^2_p(X)\cap\ker T$. The trace space is defined as:
\begin{align}\label{eq:trace_space}
E_q:=[E_1,E_0]_{1/q,q}\hookrightarrow[H^2_p(X),L_p(X)]_{_{1/q,q}}=B^{2-2/q}_{p,q}(X)\hookrightarrow C^\tau(X).
\end{align}
Here, the last embedding holds since $2-2/q-n/p>\tau>0$ by assumption. The operator $A(u_0)=-mv_0^{m-1}\Delta_{g,T}$ satisfies the assumptions of Theorem \ref{thm:main_result} as $v_0$ is strictly positive. Therefore, a suitable shift of $A(u_0)$ allows a bounded $H^\infty$-calculus and thus $A(u_0)$ has maximal  $L_q$-regularity. Maximal regularity is part of the assumptions of Theorem \ref{thm:Clement/Li}. Next, we consider the remaining assumptions of the theorem. To this end, we need the following result:
\begin{lem}\label{lem:hol_est} Let $v_0\in C^\tau(X)$ with $\re v_0\geq \delta>0$. We define:
\begin{align*}
W:=\{z\in\mathbb{C}:|z|< \|v_0\|_{C^\tau}+3\delta/4,\;\;\re z>\delta/4\}\;\;.
\end{align*}
A neighborhood $V$ of $v_0$ in $C^\tau(X)$ and a constant $C:=C(\delta,\|v_0\|_{C^\tau(X)})$ exist such that for all $f\in H^\infty(W)$ and $u,u'\in V$ the following estimates hold:
\begin{align*}
\|f(u)\|_{C^\tau(X)}&\leq C\|f\|_{L_\infty(W)}\;\;\text{and}\\
|f(u)-f(u')\|_{C^\tau(X)}
&\leq C\|f\|_{L_\infty(W)}\|u-u'\|_{C^\tau(X)}.
\end{align*}
\end{lem}

\begin{proof}
We choose $V:=B(v_0,\delta/4)$. Since all functions in $V$ are continuous, we obtain:
\begin{align*}
\image V:=\cup_{u\in V}\image u\subset W'':=\{z\in\mathbb{C}:|z|< \|v_0\|_{C^\tau}+\delta/4,\;\;\re z>\delta(1-1/4)\}.
\end{align*}
Furthermore, we define $W':=\{z\in\mathbb{C}:|z|< \|v_0\|_{C^\tau}+\delta/2,\;\;\re z>\delta(1-1/2)\}$. By definition, some distance between the boundary of $W''$ and the boundary of $W'$ exists, i.e., $d(\partial W'',\partial W')\geq\delta/4$.
Therefore, $|\eta-u(x)|\geq \delta/4$ for all $u\in V$, $\eta\in\partial W'$ and $x\in X$. It is well-known that such a lower bound implies that $(\eta-u)^{-1}\in C^\tau(X)$. Moreover, the following estimate holds:
\begin{align*}
\|(\eta-u)^{-1}\|_{C^\tau(X)}\leq 16/\delta^2\|\eta-u\|_{C^\tau(X)}\leq 16/\delta^2 (2\|v_0\|_{C^\tau(X)}+3\delta/4)=:S.
\end{align*}
We can estimate the length of the boundary:
$|\partial W'|\leq 2\pi (\|v_0\|_{C^\tau(X)}+\delta/2):=2\pi L$. For all $u\in V$ and $x\in X$, we obtain the following identity from the Cauchy integral representation:
\begin{align*}
f(u(x))=\frac{1}{2\pi i}\int_{\partial W'}f(\eta) (\eta-u(x))^{-1}\,d\eta.
\end{align*}
Thus, we obtain the first estimate $\|f(u)\|_{C^\tau(X)}\leq LS\|f\|_{H^\infty(W)}$.
For $u,u'\in V$, we use the resolvent identity to obtain:
\begin{align*}
f(u(x))-f(u'(x))=\frac{u'(x)-u(x)}{2\pi i}\int_{\partial W'} f(\eta)(\eta-u(x))^{-1}(\eta-u'(x))^{-1}\,d\eta
\end{align*}
We can estimate the $C^\tau(X)$-norm as before. Therefore, the $C^\tau(X)$-norm of the left hand side can be estimated as stated in the lemma.
\end{proof}
According to the assumptions of Theorem \ref{thm_PME} and Embedding \eqref{eq:trace_space}, the function $v_0$ satisfies the assumptions of Lemma \ref{lem:hol_est}. We choose a neighborhood $V$ of $v_0$, according to Lemma \ref{lem:hol_est}. Additionally, we choose a neighborhood $U$ of zero in $E_q$ such that the image of $U+v_0$ under the Embedding \eqref{eq:trace_space} belongs to $V$. 
For $i\in\{1,2\}$, Lemma \ref{lem:hol_est} applies to $f(z):=z^{m-i}$. Therefore:
\begin{align}
\label{eq:est_multi_1}
\|(u+v_0)^{m-i}\|_{C^\tau(X)}&\leq C\;\;\text{for all}\;\; u\in U\;\;\text{and}\\
\label{eq:est_multi_2}
\|(u+v_0)^{m-i}-(u'+v_0)^{m-i}\|_{C^\tau(X)}&\leq C\|u-u'\|_{E_q}\;\;\text{for all}\;\; u,u'\in U.
\end{align}
We recall that $C^\tau(X)\hookrightarrow\mathcal{B}(E_0)$ as multiplication operators.
Thus, Estimate \eqref{eq:est_multi_2} implies
\begin{align*}
\|A(u)-A(u')\|_{\mathcal{B}(E_1;E_0)}&\leq m\|(u+v_0)^{m-1}-(u'+v_0)^{m-1}\|_{\mathcal{B}(E_0)}\|\Delta_{g,T}\|_{\mathcal{B}(E_1;E_0)}\\
&\leq C\|u-u'\|_{E_q}
\end{align*}
for all $u,u'\in U$. Therefore, Assumption (CL1) 
in Theorem \ref{thm:Clement/Li} is satisfied.
Next, we verify Assumption (CL2). 
To this end, we define\\ $h(u)=(u+v_0)^{m-2}|\nabla (u+v_0)|^2_{g}$ and observe:
\begin{align*}
h(u)-h(u')=&(u+v_0)^{m-2}|\nabla(u+v_0)|^2_g-(u'+v_0)^{m-2}|\nabla(u'+v_0)|^2_g
\\=&\left((u+v_0)^{m-2}-(u'+v_0)^{m-2}\right)|\nabla(u+v_0)|_g^2\\&+(u'+v_0)^{m-2}\left(|\nabla(u+v_0)|^2_g-|\nabla(u'+v_0)|^2_g\right)\\
=&\left((u+v_0)^{m-2}-(u'+v_0)^{m-2}\right)|\nabla(u+v_0)|_g^2\\
&+(u'+v_0)^{m-2}\langle\nabla(u-u'),\nabla(u+v_0)\rangle_g\\
&+(u'+v_0)^{m-2}\langle\nabla(u'+v_0),\nabla(u-u')\rangle_g.
\end{align*}
The assumption $1>n/p+2/q$ and the Embedding \eqref{eq:trace_space} imply that $E_q\hookrightarrow C^1(X)$ and $E_q\hookrightarrow H^1_p(X)$. Thus, for all $u,u'\in E_q$, the following estimate holds:
\begin{align*}
\|\langle \nabla_g u,\nabla_g u'\rangle_g\|_{E_0}=\|\langle \nabla_g u,\nabla_g u'\rangle_g\|_{L_p(X)}\leq \|u\|_{C^1(X)}\| u'\|_{H^1_p(X)}\leq\|u\|_{E_q}\|u'\|_{E_q}.
\end{align*}
Therefore, for $u,u'\in U$, we obtain:
\begin{align*}
\||\nabla(u+v_0)|^2\|_{E_0}&\leq C,\\ 
\|\langle\nabla(u-u'),\nabla(u-v_0)\rangle_g\|_{E_0}&\leq C\|u-u'\|_{E_q},\;\;\text{and}\\ \|\langle\nabla(u'-v_0),\nabla(u-u')\rangle_g\|_{E_0}&\leq C\|u-u'\|_{E_q}.
\end{align*}
The Estimates \eqref{eq:est_multi_1}, \eqref{eq:est_multi_2}, and those above imply $\|h(u)-h(u')\|_{E_0}\leq C\|u-u'\|_{E_q}$. We obtain $\|((u-v_0)^{m-1}-(u-v_0)^{m-1})\Delta_gv_0\|_{E_0}\leq C\|u-u'\|_{E_q}$ for all $u,u'\in U$ from the assumption $v_0\in H^2_p(X)$ and Estimate \eqref{eq:est_multi_2}. Thus, $\|f(u)-f(u')\|_{E_0}\leq C\|u-u'\|_{E_q}$ for all $u,u'\in U$. In other words, Assumption (CL2) 
is satisfied. Therefore, Theorem \ref{thm:Clement/Li} can be applied to Problem \eqref{eq:abstract_parabolic_problem} which completes the proof of Theorem \ref{thm_PME}.

\end{document}